\documentclass{amsart}

\usepackage{graphicx,amssymb,amsmath,xcolor,enumerate}
\usepackage{hyperref}
\hypersetup{
pdffitwindow=false,     
pdfstartview={FitH},    
colorlinks=true,      
citecolor=red,
}

\usepackage{enumitem}

\usepackage{float}

\makeatletter
\def\l@subsection{\@tocline{2}{0pt}{2.5pc}{5pc}{}}
\makeatother

\usepackage{underscore}

\usepackage{geometry}
\DeclareMathOperator{\p}{p.v.}

\DeclareSymbolFont{largesymbol}{OMX}{yhex}{m}{n}
\DeclareMathAccent{\Widehat}{\mathord}{largesymbol}{"62}
\newcommand*\di{\mathop{}\!\mathrm{d}}
\newcommand{\occ}{\overline{\mathbb{C}_+}}

\numberwithin{equation}{section}              
\newtheorem{theorem}{Theorem}[section]

\newtheorem{lemma}{Lemma}[section]

\newtheorem*{proposition*}{Proposition}

\newtheorem*{corollary*}{Corollary}
\newtheorem{definition}{Definition}[section]
\newtheorem*{definitions*}{Definitions}

\newtheorem*{acknowledgements*}{Acknowledgements}

\newtheorem*{conjecture*}{\bf Conjecture}

\newtheorem*{example*}{\bf Example}
\theoremstyle{remark}
\newtheorem{remark}{\bf Remark}[section]

\begin{document}
\date{}                                     

\author{Yu Gao}
\address[Y. Gao]{Department of Applied Mathematics, The Hong Kong Polytechnic University, Hung Hom, Kowloon, Hong Kong}
\email{mathyu.gao@polyu.edu.hk}
\author{Cong Wang}
\address[C. Wang]{Department of Mathematics, Harbin Institute of Technology,  Harbin,	150001, P.R. China.}
\email{math_congwang@163.com}
\author{Xiaoping Xue}
\address[X. Xue]{Department of Mathematics, Harbin Institute of Technology, Harbin,	150001, P.R. China.}
\email{xiaopingxue@hit.edu.cn}

\title[Weak and spatial analytic solutions]{Global existence and spatial analyticity for a nonlocal flux with fractional diffusion}

\begin{abstract}
In this paper, we study a one dimensional nonlinear equation with diffusion  $-\nu(-\partial_{xx})^{\frac{\alpha}{2}}$ for $0\leq \alpha\leq 2$ and $\nu>0$. We use a viscous-splitting algorithm to obtain global nonnegative weak solutions in space $L^1(\mathbb{R})\cap H^{1/2}(\mathbb{R})$ when $0\leq\alpha\leq 2$. For subcritical $1<\alpha\leq 2$ and critical case $\alpha=1$,  we obtain global existence and uniqueness of nonnegative spatial analytic solutions. We use a fractional bootstrap method to improve the regularity of mild solutions in Bessel potential spaces for subcritical case $1<\alpha\leq 2$. Then, we show that the solutions are spatial analytic and can be extended globally. For the critical case $\alpha=1$, if the initial data $\rho_0$ satisfies $-\nu<\inf\rho_0<0$, we use the characteristics methods for complex Burgers equation to obtain a unique spatial analytic solution to our target equation in some bounded time interval. If $\rho_0\geq0$, the  solution exists globally and converges to steady state. 
\end{abstract}

\maketitle

{
\hypersetup{linkcolor=blue}
\tableofcontents
}

\section{Introduction}
In this paper, we are going to study the following nonlinear partial differential equation on the real line $\mathbb{R}$:
\begin{equation}\label{eq:Hilbert}
\left\{\begin{aligned}
&\partial_t\rho+\partial_x[\rho (u-\gamma x)]=- \nu\Lambda^{\alpha}  \rho,\quad u=H\rho,\quad t>0,~~x\in\mathbb{R},\\
&\rho(x,0)=\rho_0(x),~~x\in\mathbb{R}
\end{aligned}
\right.
\end{equation}
with $\gamma\geq0$, $\nu>0$ and $0\leq \alpha\leq 2$.  The velocity field $H\rho$ stands for the Hilbert transform of $\rho$:
\[
(H\rho)(x,t):=\frac{1}{\pi}\p\int_{\mathbb{R}}\frac{\rho(y,t)}{x-y}\di y.
\]
Here $\nu$ is a positive number called the viscosity coefficient, and it controls the strength of the dissipation term. For $0\leq \alpha\leq 2$, the fractional Laplacian $\Lambda^\alpha  \rho=(-\partial_{xx})^{\frac{\alpha}{2}}\rho$  is defined by its Fourier transform:
\[
[\mathcal{F}( \Lambda^\alpha  \rho)](\xi,t) = |\xi|^\alpha[\mathcal{F}(\rho)](\xi,t).
\]
The parameter $\alpha$ also controls the magnitude of the dissipation term.  For $\nu>0$, let the kernel $G_\alpha$ be the fundamental solution of the linear operator $\partial_t+\nu\Lambda^\alpha$, and it is defined by:
\[
\mathcal{F}(G_\alpha)(\xi,t):=e^{-\nu t|\xi|^\alpha}.
\]
We call the case $\alpha>1$, $\alpha=1$, and $\alpha<1$ of  \eqref{eq:Hilbert}  as subcritical, critical, and supercritical cases respectively. 

When $\nu=0$, equation \eqref{eq:Hilbert} becomes
\begin{equation}\label{eq:meanfield}
\partial_t\rho+ \partial_x[\rho (u-\gamma x)]=0,~~u=H\rho,~~x\in\mathbb{R},~~t>0.
\end{equation}
Equation \eqref{eq:meanfield} is the mean field equation of the following Dyson Brownian motion  \cite{rogers1993interacting,cepa1997diffusing,berman2019propagation}:
\begin{align}\label{eq:DysonBrownian}
\di \lambda_j(t)=\frac{1}{\sqrt{N}}\di B_j(t)+\frac{1}{\pi N}\sum_{k\neq j}\frac{\di t}{\lambda_j(t)-\lambda_k(t)}-\gamma\lambda_j(t)\di t,\quad 1\leq j\leq N,
\end{align}
which describes the evolution of eigenvalues $\{\lambda_j\}_{j=1}^N$ of a $N\times N$ Hermitian matrix given by matrix valued  Ornstein-Uhlenbeck (OU) process \cite{dyson1962brownian,erdos2017dynamical,Tao2012Topics}. Next, we list three important aspects of equation \eqref{eq:meanfield}.

\emph{Space-time rescaling:} For equation \eqref{eq:meanfield}, an important fact is that the the linear term $-\gamma \partial_x(x\rho)$ with $\gamma>0$ can be reformulated into the case $\gamma=0$ by the following space-time rescaling:
\begin{align}\label{eq:transform}
\tilde{\rho}(y,\tau)\sqrt{1+2\gamma\tau}=\rho(x,t),\quad x=\frac{y}{\sqrt{1+2\gamma\tau}},\quad t=\frac{1}{2\gamma}\log(1+2\gamma\tau).
\end{align}
Then, if $\rho$ be a solution to \eqref{eq:meanfield}, $\tilde{\rho}$ is a solution to equation \eqref{eq:meanfield} with $\gamma=0$:
\begin{align}\label{eq:meanfieldB}
\partial_\tau\tilde{\rho}+\partial_y(\tilde{\rho}\tilde{u})=0,\quad \tilde{u}=H\tilde{\rho}.
\end{align}
The above transformation has the same effect for  equation \eqref{eq:Hilbert} with $\alpha=2$, but not for $0<\alpha<2$. 

\emph{Gradient flow structure:} Equation \eqref{eq:Hilbert} with $\alpha=2$ or $\nu=0$ has a gradient flow structure in the probability measure space with Wasserstein distance with respect to a free energy functional given by 
\begin{align}\label{eq:interactionEnergy}
E(\rho(\cdot,t))&=\frac{\gamma}{2}\int_{\mathbb{R}} x^2\rho(x,t) \di x-\frac{1}{2}\int_{\mathbb{R}}\int_{\mathbb{R}}\log|x-y|\rho(x,t)\rho(y,t)\di x\di y+\nu\int_{\mathbb{R}}\rho(x,t)\log\rho(x,t)\di x\nonumber\\
&=:E_{\textnormal{h}}(\rho(\cdot,t))+E_{\textnormal{i}}(\rho(\cdot,t))+E_{\textnormal{e}}(\rho(\cdot,t)).
\end{align}
Here $E_{\textnormal{h}}$ is a harmonic  trap  energy,  $E_{\textnormal{i}}$ is  an interaction energy, and   $E_{\textnormal{e}}$ is the entropy. 
Then, equation \eqref{eq:Hilbert} with $\alpha=2$ or $\nu=0$ is recast to
\begin{align}\label{eq:gradientflow}
\partial_t\rho-\partial_x\left[\rho\partial_x\left(\frac{\delta E}{\delta\rho}\right)\right]=0,\quad \frac{\delta E}{\delta\rho}=  \frac{\gamma}{2} x^2-\int_{\mathbb{R}}\log|x-y|\rho(y,t)\di y.
\end{align}
By the properties of this gradient flow structure, Carrillo et. al. \cite{carrillo2012mass} obtained the existence and uniqueness of global probability measure solutions. They also proved $\gamma$-convexity along Wasserstein geodesics of the energy and hence obtained exponential convergence to the steady state given by Wigner's semicircle law
\begin{align}\label{eq:semicirclelaw}
\mu(\di x)= \rho(x) \di x := \frac{\sqrt{(4-x^2)_+}}{2\pi}\di x.
\end{align}

\emph{Complex Burgers equation:} From equation \eqref{eq:meanfield}, if we define $g(x,t)=H\rho(x,t)-i\rho(x,t)-\gamma x$, then the analytical extention of $g$ on the upper half complex plane $\mathbb{C}_+:=\{z:\Im(z)=\mathrm{Im}(z)>0\}$ satisfies the following complex Burgers equation with a force term $\gamma^2z$ \cite{castro2008global,gao2020large}:
\begin{align}\label{eq:complexBurgers0}
\partial_tg+g\partial_zg=\gamma^2z,\qquad z\in\mathbb{C_+},~~t>0.
\end{align}
When $\gamma=0$, Castro and C\'ordoba \cite{castro2008global} proved global (in time) existence and uniqueness of spatial analytic solutions ($t>0$) to \eqref{eq:meanfield} with strictly positive initial data $0<\rho_0\in L^2(\mathbb{R})\cap C^{0,\delta}(\mathbb{R})$  via characteristics method for \eqref{eq:complexBurgers0}. However, if there is $x_0\in\mathbb{R}$ such that  $\rho(x_0)=0$, then the solution $ \rho$ will blow up in $H^s(\mathbb{R})$ for  $s>\frac{3}{2}$ in finite time \cite{castro2008global}. These two results hold also for $\gamma>0$ due to the rescaling \eqref{eq:transform}, and the  global solutions with $\rho_0>0$ converge to the steady state pointwisely \cite{gao2020large}. Global nonnegative weak solutions in $L^{\infty}(0,T;  L^1(\mathbb{R})\cap H^{\frac12}(\mathbb{R}))$ to \eqref{eq:meanfield} were also obtained \cite{gao2020large}. 

In this paper, we are going to study  equation \eqref{eq:Hilbert} with $0\leq \alpha\leq 2$. We first use a viscous-splitting algorithm (see, e.g. \cite[Chapter 3]{majda2002vorticity}) to obtain global weak solutions (see Definition \ref{def:weak1}) for the whole range  $0\leq \alpha\leq 2$. The following theorem is obtained:
\begin{theorem}\label{thm:weaktheorem}
Assume $0\leq \rho_0\in L^1(\mathbb{R})\cap H^{1/2}(\mathbb{R})$. Then, there exists a global nonnegative weak solution to \eqref{eq:Hilbert} satisfying
\[
\rho\in L^\infty(0,\infty; L^1(\mathbb{R})\cap H^{1/2}(\mathbb{R}))\cap W^{1,\infty}(0,\infty;H^{-3}(\mathbb{R})).
\]
Moreover, we have:
\begin{align}\label{eq:solutionestimate}
\|\rho(t)\|_{L^1}\leq \|\rho_0\|_{L^1},\quad \|\rho(t)\|_{H^{1/2}}\leq  \|\rho_0\|_{H^{1/2}},~~t>0.
\end{align}
\end{theorem}
The reason to choose viscous-splitting algorithm is simple, because both fractional heat equation and equation \eqref{eq:meanfield} yield global analytic solutions preserving positivity and norms of $L^1(\mathbb{R})$ and $H^{1/2}(\mathbb{R})$ for positive initial data. Hence, we only need to use some compactness argument to derive global weak solutions (see Section \ref{sec:weak}). 
In papers \cite{bae2018global}, global weak solutions to the following general models were studied:
\begin{align}\label{eq:general}
\partial_t\rho+\partial_x\rho H\rho+\delta \rho\partial_xH\rho=- \nu\Lambda^{\alpha}  \rho.
\end{align}
When $\delta=1$, the above equation becomes equation \eqref{eq:Hilbert}. For different range of $\alpha$ and $\delta$, they obtained several results about global weak solutions to equation \eqref{eq:general}. Among these results, \cite[Theorem 1.1]{bae2018global} shares some similarities with Theorem \ref{thm:weaktheorem} in this paper.
For strictly positive initial data $\rho_0\in L^1(\mathbb{R})\cap H^{1/2}(\mathbb{R})$, global weak solutions were obtained in \cite[Theorem 1.1]{bae2018global} for supercritical case $0<\alpha<1$ and $\delta\geq \frac{1}{2}.$  In comparison, we do not need the strictly positive assumption for initial data and weak solutions are obtained for all $0\leq \alpha\leq 2$. 

We will obtain spatial analytic solutions for subcritical and critical cases $1\leq \alpha\leq 2$ by different methods.
For the subcritical case $1<\alpha\leq 2$, we have the following theorem:
\begin{theorem}\label{thm:subcritical}
Let $0\leq \rho_0\in  L^{\frac{1}{\alpha-1}}(\mathbb{R})$. Then, there is a unique nonnegative solution $\rho\in C([0,\infty);L^{\frac{1}{\alpha-1}}(\mathbb{R}))\cap C^\infty((0,\infty);H^{\theta,q}(\mathbb{R}))$ for any $\theta\geq 0$ and $\frac{1}{\alpha-1}\leq q\leq \infty$. Moreover, we have
\begin{align*}
\|\rho(t)\|_{\dot{H}^{\theta,q}(\mathbb{R})}\leq Ct^{-\frac{\theta}{\alpha}-1+\frac{1}{\alpha}(1+\frac{1}{q})},\quad  \frac{1}{\alpha-1}\leq q\leq \infty,~~t>0,
\end{align*}
and 
\begin{align*}
\|\partial_x^n\rho(t)\|_{L^q}\leq K^{n}n^{n}t^{-\frac{n}{\alpha}-1+\frac{1}{\alpha}(1+\frac{1}{q})},\quad\forall n\in\mathbb{N},~~t>0,
\end{align*}
for some constant $K$ independent of $n$. Consequently, $\rho(\cdot,t)$ is spatially analytic for $t>0$
\end{theorem}
Here $H^{\theta,q}$ denotes Bessel potential space (or fractional Sobolev spaces, see Section \ref{sec:bessel}). Our strategies to prove Theorem \ref{thm:subcritical} is as follows. We consider the mild solutions to \eqref{eq:Hilbert} ($1<\alpha\leq 2$) of the form:
\begin{align}\label{eq:mildsolution}
\rho(x,t)=G_\alpha(\cdot,t)\ast\rho_0-\int_0^t \partial_xG_\alpha(\cdot,t-s)\ast(\rho(s) H\rho(s)) \di s.
\end{align}
Notice that if $\rho(x,t)$ is a solution to \eqref{eq:Hilbert}  with initial data $\rho_0$, then $\rho_\lambda(x,t)=\lambda^{\alpha-1}\rho(\lambda x,\lambda^\alpha t)$ is also a solution with initial data $\rho_{\lambda,0}(x)=\lambda^{\alpha-1}\rho_0(\lambda x)$. This scaling preserves the $L^{\frac{1}{\alpha-1}}(\mathbb{R})$ norm. It is nature to study mild solutions to \eqref{eq:Hilbert} with initial data $\rho_0\in L^{\frac{1}{\alpha-1}}(\mathbb{R})$. Next, we describe the results for the subcritical case $1<\alpha\leq 2$ in several steps.

\emph{Local existence and uniqueness:} When $\rho_0\in L^{\frac{1}{\alpha-1}}(\mathbb{R})$, we use Banach fixed point theorem to prove local existence and uniqueness of mild solutions in the following Banach space (see Theorem \ref{thm:contraction}):
\begin{align}\label{eq:space1}
X_T:=\left\{f\in C_b((0,T];L^{\frac{1}{\alpha-1}}(\mathbb{R})),\quad \sup_{0<t\leq T}t^{\frac{\alpha-1}{2\alpha}}\|f(t)\|_{L^{\frac{2}{\alpha-1}}}<\infty  \right\},
\end{align}
with norm
\[
\|f\|_{X_T}:=\max\left\{\sup_{0<t\leq T}\|f(t)\|_{L^{\frac{1}{\alpha-1}}},\quad \sup_{0<t\leq T}t^{\frac{\alpha-1}{2\alpha}}\|f(t)\|_{L^{\frac{2}{\alpha-1}} }\right \}.
\]
The idea to choose the above space for contraction argument is well-known. One can refer to \cite{weissler1980local,kato1984strongl,biler1995cauchy} for some variations of this method for local existence of solutions to different equations. 

\emph{Fractional bootstrapping for regularity in Bessel potential spaces:} We improve the spatial regularity of solution $\rho$ by  a fractional bootstrapping and obtain time decay estimate in Bessel potential spaces (see Theorem \ref{thm:regularity}). Here, we adopt the name ``fractional bootstrapping" used in \cite{dong2009optimal} for fractional Navier Stokes equations, although the proof of high order regularity and spacial analyticity are different. We first show the hyper-contractivity and prove that $\rho(t)\in L^q(\mathbb{R})$ for any $\frac{1}{\alpha-1}\leq q\leq\infty$ and $t>0.$ In comparison with the usual method for hyper-contractivity (see, e.g., \cite{bedrossian2014existence,liu2016refined,dong2008spatial,li2011one}), the proof is more direct in the sense that we do not need any kind of a priori estimates or contraction argument. Then, we improve the spatial regularity of the mild solutions step by step. From the proof of Theorem \ref{thm:regularity}, we see that in each step, the time integral in the nonlinear term of mild solutions only allows us to increase spatial regularity by  some decimal order $0<\ell<\alpha-1\leq 1$. Hence, Bessel potential spaces are nature choices for this method and  this is also the reason for the name fractional bootstrapping. 

\emph{Nonnegativity:} For the nonnegativity of mild solutions, we follow the method in \cite[Lemma 2.7]{li2011one}, but without the condition $0\leq \rho_0\in L^{\frac{1}{\alpha-1}}(\mathbb{R})\cap L^p(\mathbb{R})$ for some $\frac{1}{\alpha-1}<p<\infty$. In other words, we only need $0\leq \rho_0\in L^{\frac{1}{\alpha-1}}(\mathbb{R}),$ which is more compatible with the results for existence and regularity.

\emph{Spatial analyticity and global extension:} To prove the spatial analyticity of mild solutions, we are going to give a simple generalization of the method in \cite{giga2002regularizing,sawada2005analyticity} for Navier-Stokes equation. As noticed in \cite[Remark 7]{dong2008spatial} (or \cite[Remark 2.4]{dong2009optimal}), if we directly use the method in \cite{giga2002regularizing,sawada2005analyticity}, we will only obtain
\[
|\partial_x^n\rho(x,t)|\leq K^{n+1}n^{2n/\alpha}t^{-\frac{n}{\alpha}-1+\frac{1}{\alpha}}
\]
for some constant $K$ independent of $n\in\mathbb{N}$. This does not imply the spatial analyticity of $\rho$ if $\alpha<2$. To overcome this difficulty, we  are going to improve the method in \cite{giga2002regularizing,sawada2005analyticity} and use it for fractional diffusion with $1<\alpha\leq 2$ (see Theorem \ref{thm:spatialanaly}). Then, by the $L^{p}(\mathbb{R})$ maximum principle for the nonnegativity solutions, we extend the solutions globally (see Lemma \ref{lmm:MaxP} and Theorem \ref{thm:global}). 

Notice that there is another smart way for  spatial analyticity given by Dong and Li \cite{dong2008spatial}, where some spaces involving the information of high order derivatives were introduced for contraction argument to obtain spacial analytic solutions to the subcritical dissipative quasi-geostrophic equations. By the same method, Li and Rodrigo \cite{li2011one} studied local existence and finite time blow-up behavior of solutions for the following equation with $1<\alpha\leq 2$: 
\begin{equation}\label{eq:Hilbert2}
\left\{\begin{aligned}
&\partial_t\rho-\partial_x(\rho H\rho)=-\nu \Lambda^\alpha \rho,~~t>0,~~x\in\mathbb{R},\\
&\rho(x,0)=\rho_0(x),~~x\in\mathbb{R},
\end{aligned}
\right.
\end{equation}
Comparing with equation \eqref{eq:Hilbert}, the sign of the nonlinear term is different. To study the nonnegative solutions of \eqref{eq:Hilbert} is equivalent to study the non-positive solutions of \eqref{eq:Hilbert2}. For nonnegative solutions of \eqref{eq:Hilbert}, we have the $L^p(\mathbb{R})$ maximum principle  to extend the mild solutions globally, which is false for nonnegative solutions of \eqref{eq:Hilbert2}. And finite time blow-up behavior of solutions to \eqref{eq:Hilbert2} with some special initial data was proved by \cite[Theorem 3.1]{li2011one}. The reason for this difference can be easily observed from the particle systems for these two equations. Formally, equation \eqref{eq:Hilbert2}  with $\nu=0$ corresponds to the mean field equation for the following particle system:
\begin{align}\label{eq:DysonBrownian1}
\di \lambda_j(t)=\frac{1}{\sqrt{N}}\di B_j(t)-\frac{1}{\pi N}\sum_{k\neq j}\frac{\di t}{\lambda_j(t)-\lambda_k(t)},\quad 1\leq j\leq N.
\end{align}
The force between particles is attractive force. Hence, they try to aggregate together to form singularities. The force between particles in \eqref{eq:DysonBrownian} is repulsive force, and global well-posedness can be obtained; see \cite{rogers1993interacting} for global well posedness of system \eqref{eq:DysonBrownian}. 

For the critical case $\alpha=1$ of \eqref{eq:Hilbert}, we will also prove global existence and uniqueness of spatial analytic solutions. Due to the following relation
\[
\Lambda\rho = \partial_xH\rho = H\partial_x\rho = \frac{1}{\pi}\p\int_{\mathbb{R}}\frac{\rho(x,t)-\rho(y,t)}{|x-y|^2}\di y = \partial_xu,
\]
equation \eqref{eq:Hilbert} is rewritten as
\begin{equation}\label{eq:Hilbert1}
\partial_t\rho+\partial_x\left[\rho (u-\gamma x)+\nu u\right]=0.
\end{equation}
When $\gamma=0$, global spatial analytic solutions to \eqref{eq:Hilbert1} were given by \cite[Theorem 4.1]{castro2008global} for initial data $\rho_0>-\nu$ using the same method as the case for $\nu=0$ (described before). If $\rho_0\geq-\nu$ and there exists $x_0\in\mathbb{R}$ such that $\rho_0(x_0)=\nu$, then $\partial_xH\rho$ will blow up in finite time \cite[Thoerem 4.8]{castro2008global}. Comparing with the cases $\nu=0$ or $\alpha=2$ of \eqref{eq:Hilbert}, the transformation \eqref{eq:transform} does not work for equation \eqref{eq:Hilbert1} with $\nu>0$. Hence, we can not derive the spatial analytic solutions to \eqref{eq:Hilbert1} with $\gamma>0$  directly from the results \cite[Theorem 4.1]{castro2008global}  by transformation \eqref{eq:transform}. Moreover, equation \eqref{eq:Hilbert1} also does not have gradient flow structure as \eqref{eq:gradientflow}. In this paper, we will use a similar idea as \cite[Theorem 4.1]{castro2008global} to obtain spatial analytic solutions to \eqref{eq:Hilbert1}  with $\gamma>0,$ and the solutions show some different and interesting properties in comparison with the case $\gamma=0$. We have the following theorem:
\begin{theorem}\label{thm:critical}
Let $0\leq \mu<\nu$ and $-\mu\leq \rho_0\in  L^1(\mathbb{R})\cap H^s(\mathbb{R})$ with $s>1/2$. Denote $T=\frac{1}{\gamma}\ln(\frac{2\nu}{\mu}-1)$. Then, there exists a unique spatial analytic solution $\rho(x,t)$ to \eqref{eq:Hilbert} with $\alpha=1$ in $(0,T)$.

When $\mu=0$, the solution $\rho(x,t)$ exists globally and converges to the steady state given by semicircle law:
\begin{align}\label{eq:steadypho}
\lim_{t\to\infty}\rho(x,t)=\rho_\infty(x):=\frac{\sqrt{\sqrt{[\gamma^2x^2-\nu^2-2\gamma]^2+4\gamma^2x^2\nu^2}-[\gamma^2x^2-\nu^2-2\gamma]}-\sqrt{2}\nu}{\sqrt{2}\pi}.
\end{align}

\end{theorem}
As shown in the above theorem, we obtain global spatial analytic solutions to \eqref{eq:Hilbert1} when initial data $\rho_0\geq0$, and  the solutions converges to steady state pointwisely. However, if $\rho_0\geq-\mu$ for some $0\leq \mu<\nu$, we can only obtain spatial analytic solutions in time interval $(0,T)$ with $T=\frac{1}{\gamma}\ln(\frac{2\nu}{\mu}-1)$, which is different with the case $\gamma=0$ given by \cite[Theorem 4.1]{castro2008global}.

The rest of this paper is organized as follows. We are going to use a viscous-splitting algorithm to prove Theorem \ref{thm:weaktheorem} in the next section.
  Local existence and uniqueness of mild solutions to \eqref{eq:Hilbert} with $1<\alpha\leq 2$ are obtained in Section \ref{sec:local}. Then, we improve the regularity and show the spatial analyticity of solutions in Section \ref{sec:regularity}. In Section \ref{sec:global} we extend the local solution globally by the $L^p$ maximum principle for nonnegative solutions. For the critical case $\alpha=1$, we first obtain global $\occ$-holomorphic solutions to the corresponding Complex Burgers equation (see \eqref{eq:complexBurgers}) in Section \ref{sec:analytic}. Then, we use these $\occ$-holomorphic solutions  to recover the solutions to \eqref{eq:Hilbert} with $\alpha=1$, and derive the pointwise convergence to the steady state when $\rho_0\geq0$.

\section{Global nonnegative weak solutions for  $0\leq\alpha\leq 2$}\label{sec:weak}
In this section, we are going to use a viscous-splitting algorithm to obtain global weak solutions in $L^1(\mathbb{R})\cap \dot{H}^{1/2}(\mathbb{R})$ to equation \eqref{eq:Hilbert} with $0\leq\alpha\leq 2$. Here, we only consider the case $\gamma=0$. Note that we have interpolation inequality
\[
\|\rho\|_{L^2}\leq 3\|\rho\|_{L^1}^{1/2}\|\rho\|^{1/2}_{\dot{H}^{1/2}}.
\]
Hence $\rho \in L^1(\mathbb{R})\cap \dot{H}^{1/2}(\mathbb{R})$ is equivalent to $\rho \in L^1(\mathbb{R})\cap H^{1/2}(\mathbb{R})$. Let us define the weak solutions:
\begin{definition}\label{def:weak1}
For $T>0$ and $0\leq \rho_0\in L^1(\mathbb{R})\cap H^{1/2}(\mathbb{R})$, a nonnegative function $\rho\in L^\infty(0,T;L^1(\mathbb{R})\cap H^{1/2}(\mathbb{R}))\cap W^{1,\infty}(0,T;H^{-m}(\mathbb{R}))$ for some $m>0$  is said to be a weak solution of  equation \eqref{eq:Hilbert} if
\begin{multline}\label{eq:defweak1}
\int_0^T\int_{\mathbb{R}}\partial_t\phi(x,t)\rho( x,t) \,\di x\di t+\int_{\mathbb{R}}\phi(x,0)\rho_0(x)\di x\\
=-\frac{1}{2}\int_0^T\int_{\mathbb{R}}\int_{\mathbb{R}}\frac{\partial_x\phi(x,t)-\partial_x\phi(y,t)}{x-y}\rho(x,t)\rho(y,t)\,\di x\di y\di t\\
+\nu\int_0^T\int_{\mathbb{R}}\rho(x,t)\Lambda^\alpha\phi(x,t)\di x\di t,
\end{multline}
holds for any test function $\phi\in C_c^\infty(\mathbb{R}\times[0,T))$.
\end{definition}

Next, we describe the viscous-splitting algorithm by means of a Trotter formula.
Denote the solution operator to Dyson equation by $D(t)$, such that $\omega(x,t)=D(t)\rho_0(x)$ solves 
\begin{equation}\label{eq:Dyson}
\left\{\begin{aligned}
&\partial_t\omega+\partial_x(\omega H\omega)=0,~~x\in\mathbb{R},~~t>0,\\
&\omega(x,0)=\rho_0(x).
\end{aligned}\right.
\end{equation}
Also denote $G_\alpha(t)\omega_0(x):=G_\alpha(\cdot,t)\ast \omega_0$, so that $v(x,t)=G_\alpha(t)\omega_0(x)$ solves the fractional heat equation
\begin{equation}\label{eq:fracHeat}
\left\{\begin{aligned}
&\partial_tv=-\nu\Lambda^\alpha v,~~x\in\mathbb{R},~~t>0,\\
&v(x,0)=\omega_0(x).
\end{aligned}\right.
\end{equation}
Let $\varphi_h>0$ ($h>0$) be the standard Friedrichs mollifier. Set
\begin{align}\label{eq:initialmollifier}
\rho_{0,h}=\rho_0\ast\varphi_h.
\end{align}
Then, for nontrival initial datum $0\leq \rho_0\in L^1(\mathbb{R})\cap H^{1/2}(\mathbb{R})$, we have $\rho_{0,h}(x)>0$ for $x\in\mathbb{R}$ and $\rho_{0,h}\in L^1(\mathbb{R})\cap H^s(\mathbb{R})$ ($s>1/2$).
Then, the viscous-splitting algorithm by means of a Trotter formula is given by
\begin{equation}\label{eq:trotter}
\rho_{n,h}(x)=[G_\alpha(h)D(h)]^n\rho_{0,h}(x),\quad x\in\mathbb{R},
\end{equation}
where $\rho_n$ is the approximate value of the solution at time $t_n:=nh$ and $h$ is the length of time step. One can also use the Strang's method; see \cite[Chapter 3]{majda2002vorticity}. Define
\begin{equation}\label{eq:tilderho}
\tilde{\rho}_h(x,t)=D(s)\rho_{n,h}(x),\quad\rho_h(x,t)=G_\alpha(s)D(s)\rho_{n,h}(x)=G_\alpha(s)\tilde{\rho}_h(x,t)
\end{equation}
for $t=s+t_n,~~0\leq s\leq h,~~n\in\mathbb{N}.$
For $t\in(t_n,t_{n+1})$, we have
\begin{equation}\label{eq:rhoh}
\partial_t \rho_h =-\nu \Lambda^\alpha \rho_{h}-G_\alpha(s)\partial_x(\tilde{\rho}_{h}H\tilde{\rho}_{h}).
\end{equation}
Hence, for $\phi\in C_c^\infty(\mathbb{R}\times[0,T))$, we have
\begin{multline}\label{eq:tn}
\int_{t_n}^{t_{n+1}}\int_{\mathbb{R}} \rho_h\partial_t\phi\di x\di t-\int_{\mathbb{R}} \rho_h\phi\di x\Big|_{t_n}^{t_{n+1}}\\
=-\frac{1}{2}\int_{t_n}^{t_{n+1}}\int_{\mathbb{R}}\int_{\mathbb{R}}\frac{\partial_x[G_\alpha(s)\phi(t)](x)-\partial_x[G_\alpha(s)\phi(t)](y)}{x-y}\tilde{\rho}_{h}(x,t)\tilde{\rho}_{h}(y,t)\,\di x\di y\di t\\
+\nu\int_{t_n}^{t_{n+1}}\int_{\mathbb{R}}\rho_{h}\Lambda^\alpha\phi \di x\di t.
\end{multline}
Assume $T\in (t_{N_h-1},t_{N_h})$ for some positive integer $N_h$. Sum \eqref{eq:tn} for $n=1,\cdots,N_h$ together and we obtain
\begin{multline}\label{eq:approximateweak}
\int_{0}^{T}\int_{\mathbb{R}} \rho_h\partial_t\phi\di x\di t+\int_{\mathbb{R}} \rho_{0,h}\phi(x,0)\di x\\
=-\frac{1}{2} \int_{0}^{T}\int_{\mathbb{R}}\int_{\mathbb{R}}\frac{\partial_x[G_\alpha(s_t)\phi(t)](x)-\partial_x[G_\alpha(s_t)\phi(t)](y)}{x-y}\tilde{\rho}_{h}(x,t)\tilde{\rho}_{h}(y,t)\,\di x\di y\di t\\
+\nu\int_{0}^{T}\int_{\mathbb{R}}\rho_{h}\Lambda^\alpha\phi \di x\di t,
\end{multline}
where $s_t=t-t_n$ for $n$ satisfying $t\in(t_n,t_{n+1})$. Hence, $s_t\to0$ as $h\to0$.

Next, we show some compactness results for $\{\tilde{\rho}_h\}_{h>0}$ and $\{\rho_h\}_{h>0}$. We have:
\begin{lemma}\label{lmm:uniformestimates}
Assume $0\leq \rho_0\in L^1(\mathbb{R})\cap H^{1/2}(\mathbb{R})$. Let $\rho_{0,h}$ be defined by \eqref{eq:initialmollifier} for $h>0$. Then,  we have
\begin{align}\label{eq:estiamte1}
\|\rho_{n,h}\|_{L^1}\equiv \|\rho_0\|_{L^1},\quad \|\rho_{n,h}\|_{H^{1/2}}\leq  \|\rho_{0}\|_{H^{1/2}},\quad\forall n\in\mathbb{N}, ~~h>0,
\end{align}
where $\rho_{n,h}$ is given by \eqref{eq:trotter}. Moreover, we have
\begin{align}\label{eq:uniform1}
\|\tilde{\rho}_{h}(t)\|_{L^1}\equiv \|\rho_0\|_{L^1},\quad \|\tilde{\rho}_{h}(t)\|_{H^{1/2}}\leq \|\rho_{0}\|_{H^{1/2}},
\end{align}
\begin{align}\label{eq:uniform2}
\|\rho_{h}(t)\|_{L^1}\equiv \|\rho_0\|_{L^1},\quad \|\rho_{h}(t)\|_{H^{1/2}}\leq \|\rho_{0}\|_{H^{1/2}},
\end{align}
and
\begin{align}\label{eq:uniform3}
\|\partial_t\tilde{\rho}_{h}(t)\|_{H^{-3}}\leq C,\quad \|\partial_t \rho_{h}(t)\|_{H^{-3}}\leq C
\end{align}
for any $t\geq0$, $\tilde{\rho}_h$, $\rho_h$ given by \eqref{eq:tilderho} and some constant $C$ independent of $h$.
\end{lemma}
\begin{proof}
Because $0<\rho_{0,h}\in L^1(\mathbb{R})\cap H^s(\mathbb{R})$ $(s>1/2)$, we have a unique global nonnegative smooth solution to \eqref{eq:Dyson}; see \cite[Theorem 2.1]{gao2020large}. Moreover, from the proof of \cite[Theorem 2.2]{gao2020large}, we have
\[
\|D(h)\rho_{0,h}\|_{L^1}\equiv \|\rho_0\|_{L^1},\quad \|D(h)\rho_{0,h}\|_{H^{1/2}}\leq \|\rho_{0}\|_{H^{1/2}}.
\]
Since the solution to fractional heat equation \eqref{eq:fracHeat} is also nonnegative for nonnegative initial data, and it also conserves norms of $L^1(\mathbb{R})$ and $H^{1/2}(\mathbb{R})$,  inequalities in \eqref{eq:estiamte1} hold for $n=1$. By the definition of $\rho_{n,h}$, we know that  \eqref{eq:estiamte1} holds for any $n\in\mathbb{N}$. By the same reason, we also have \eqref{eq:uniform1} and \eqref{eq:uniform2}.

Next, we sketch the proof of \eqref{eq:uniform3}. Since $\tilde{\rho}_h(x,t)$ satisfies \eqref{eq:Dyson} with initial data $\rho_{n,h}(x)$ for $t\in (t_n,t_{n+1})$, the following estimate holds for any $\phi\in C_c^\infty(\mathbb{R})$:
\begin{align*}
\int_{\mathbb{R}}\phi(x)\partial_t\tilde{\rho}_h(x,t)\di x =&-\frac{1}{2} \int_{\mathbb{R}}\int_{\mathbb{R}}\frac{\partial_x\phi(x)-\partial_x\phi(y)}{x-y}\tilde{\rho}_h(x,t)\tilde{\rho}_h(y,t)\di x\di y\\
\leq& C\|\rho_0\|^2_{L^1}\|\partial_{xx}\phi\|_{L^\infty}\leq C\|\phi\|_{H^3},
\end{align*}
and hence
\begin{align*}
\|\partial_t\tilde{\rho}_h\|_{L^\infty(0,\infty;H^{-3}(\mathbb{R}))}\leq C, \quad 
\partial_t\tilde{\rho}_h\in L^\infty(0,\infty;H^{-3}(\mathbb{R})).
\end{align*}
The estimate for $\partial_t\rho_h$ in \eqref{eq:uniform3} can be obtained similarly via \eqref{eq:rhoh}.

\end{proof}
With Lemma \ref{lmm:uniformestimates}, we prove Theorem \ref{thm:weaktheorem}.

\begin{proof}[Proof of Theorem \ref{thm:weaktheorem}]
Let $h>0$ and $t\in(t_n,t_{n+1})$. From the definition \eqref{eq:tilderho}, we have
\[
\|\rho_h(t)-\tilde{\rho}_h(t)\|_{L^1}=\|G_\alpha(t-t_n)\tilde{\rho}_h(t)-\tilde{\rho}_h(t)\|_{L^1}\to 0~\textrm{ as }~h\to0.
\]
The proof of the above convergence result is the same as the estimate \eqref{eq:sindpendent} in Appendix \ref{apd:continuity}.
Hence, from Lemma \ref{lmm:uniformestimates}, there exist subsequences of $\{\tilde{\rho}_h\}_{h>0}$ and $\{\rho_h\}_{h>0}$ (still denoted as $\{\tilde{\rho}_h\}_{h>0}$ and $\{\rho_h\}_{h>0}$) such that they converge to a same function $\rho\in L^\infty(0,T; H^{1/2}(\mathbb{R}))\cap W^{1,\infty}(0,T;H^{-3}(\mathbb{R}))$:
\[
\tilde{\rho}_h, ~~\rho_h\overset{\ast}{\rightharpoonup} \rho~\textrm{ in }~L^\infty(0,T;H^{1/2}(\mathbb{R}))~\textrm{ as }~h\to0,
\]
and
\[
\partial_t\tilde{\rho}_h, ~~\partial_t\rho_h\overset{\ast}{\rightharpoonup} \partial_t\rho~\textrm{ in }~L^\infty(0,T;H^{-3}(\mathbb{R}))~\textrm{ as }~h\to0.
\]
Combining Lemma \ref{lmm:uniformestimates} and Aubin-Lions Lemma, we also know
\[
\tilde{\rho}_h,~~{\rho}_h\to \rho~\textrm{ in }~L^\infty(0,T;L_{loc}^2(\mathbb{R}))~\textrm{ as }~h\to0,
\]
and as a consequence of H\"older inequality on compact set, we have
\begin{equation}\label{eq:strong2}
\tilde{\rho}_h,~~{\rho}_h\to \rho~\textrm{ in }~L^\infty(0,T;L_{loc}^1(\mathbb{R}))~\textrm{ as }~h\to0.
\end{equation}
Hence, we have \eqref{eq:solutionestimate}.

Notice that $\partial_x[G_\alpha(s_t)\phi(t)](x)\to \partial_x\phi(x,t)$ as $h\to0$ for any $x\in\mathbb{R}$ and $t\in[0,T]$.
By the strong convergence of $\tilde{\rho}_h,~~{\rho}_h$ in \eqref{eq:strong2}, we can take the limit as $h\to0$ in \eqref{eq:approximateweak} and conclude that $\rho$ satisfies \eqref{eq:defweak1}. Hence, $\rho$ is a global weak solution to \eqref{eq:Hilbert}.

\end{proof}

\begin{remark}
Notice that we do not have uniqueness of weak solutions. For $\alpha=0, ~2$, one can use the convexity along Wasserstein geodesics of the energy \eqref{eq:interactionEnergy} to show uniqueness (see \cite{carrillo2012mass}).

Since both the fractional heat equation and \eqref{eq:Dyson} give spatial analytic solutions for strictly positive solutions, there is a high chance to obtain spatial analytic solutions for strictly positive initial data. Unfortunately, we do not have good a priori estimates to get better compactness results.
\end{remark}

\section{Global spatial analytic solutions for the subcritical case $1<\alpha\leq 2$}\label{sec:bessel}
In this section, we are going to obtain global  spatial analytic solutions to \eqref{eq:Hilbert} with $1<\alpha\leq 2$ and initial data $\rho_0\in L^{\frac{1}{\alpha-1}}(\mathbb{R})$. For simplicity, we only consider $\gamma=0$. The same results can be obtained for $\gamma>0$. First, let us introduce  Bessel potential spaces and give some properties of fractional heat kernel $G_\alpha$. For more details about Bessel potential spaces, one can refer to \cite[Chpater 6]{grafakos2009modern}. 

Similarly to the fractional Laplacian, the Bessel potential $(I-\partial_{xx})^{\ell/2} $ and Riesz potential $\Lambda^\ell :=(-\partial_{xx})^{\ell/2} $ for $\ell\in\mathbb{R}$ are defined via the Fourier transforms:
\[
[\mathcal{F}( (I-\partial_{xx})^{\ell/2} \rho)](\xi,t) = (1+|\xi|^2)^{\ell/2}[\mathcal{F}(\rho)](\xi,t),\quad [\mathcal{F}( \Lambda^\ell \rho)](\xi,t) = |\xi|^\ell[\mathcal{F}(\rho)](\xi,t).
\]
For $\ell\geq0$ and $1\leq q\leq \infty$, the Bessel potential spaces are defined by
\begin{align*}
H^{\ell,q}(\mathbb{R}):=\{f\in \mathcal{S}'(\mathbb{R}),~~(I-\partial_{xx})^{\ell/2} f\in L^q(\mathbb{R})\}=\{f\in L^q(\mathbb{R}),~~\Lambda^\ell f\in L^q(\mathbb{R})\},
\end{align*}
where $\mathcal{S}'(\mathbb{R})$ stands for the space of tempered distributions.
When $\ell$ is a positive integer and $1<q<\infty$, $H^{\ell,q}(\mathbb{R})$ coincides with the usual Sobolev spaces $W^{\ell,q}(\mathbb{R})$. For $f\in H^{\ell,q}(\mathbb{R})$, the homogeneous semi-norm is given by
\[
\|f\|_{\dot{H}^{\ell,q}}=\|\Lambda^\ell f\|_{L^q}.
\]

In the rest of this paper, for $\ell\geq 0$ we denote
\begin{align}\label{eq:constant}
C_{\ell,\alpha}:= \max\left\{\sup_{1\leq p\leq \infty}\|\Lambda^\ell G_\alpha(\cdot,1)\|_{L^p},~~\sup_{1\leq p\leq \infty}\|\Lambda^\ell \partial_xG_\alpha(\cdot,1)\|_{L^p}\right\}.
\end{align}
According to \cite[Lemma 2.2]{miao2008well}, we have
\[
|\Lambda^\ell G_\alpha(\cdot,1)|\leq C(1+|x|)^{-1-\ell},~~|\Lambda^\ell \partial_xG_\alpha(\cdot,1)|\leq C(1+|x|)^{-2-\ell},
\]
which implies $C_{\ell,\alpha}<\infty$. Later on we will use $C_{0,\alpha}$ and $C_{1,\alpha}$ for $\ell=0$ and $\ell=1$ separately. Moreover, since $\Lambda\sim \partial_x$, we also use $C_{1,\alpha}$ as the upper bound for $\sup_{1\leq p\leq \infty}\| \partial_x^2G_\alpha(\cdot,1)\|_{L^p}$.

We have the following useful estimates for fractional heat kernel. Although the proofs for similar estimates can be found in other papers (e.g., \cite{carrillo2008asymptotic,dong2008spatial}), we will give a complete proof here.
\begin{lemma}\label{lmm:key}
Let $f\in L^p(\mathbb{R})$ for $p\geq 1$. Assume $k\in\mathbb{N}$, $\ell\geq 0$ and $1\leq p\leq q\leq +\infty$. We have the following estimates:
\begin{align}\label{eq:fracdecay0}
\|\Lambda^{k\ell} G_\alpha(\cdot,t)\|_{L^q}\leq C^k_{\ell,\alpha}k^{\frac{k\ell+1}{\alpha}-\frac{1}{\alpha q}} t^{-\frac{k\ell+1}{\alpha}+\frac{1}{\alpha q}},
\end{align}
\begin{align}\label{eq:fracdecay00}
\|\Lambda^{k\ell} \partial_x G_\alpha(\cdot,t)\|_{L^q}\leq C^{k}_{\ell,\alpha}k^{\frac{k\ell+2}{\alpha}-\frac{1}{\alpha q}} t^{-\frac{k\ell+2}{\alpha}+\frac{1}{\alpha q}},
\end{align}
\begin{align}\label{eq:fracdecay1}
\|[\Lambda^{k\ell} G_\alpha(\cdot,t)]\ast f\|_{L^q}\leq C^k_{\ell,\alpha}k^{\frac{k\ell}{\alpha}+\frac{1}{\alpha}(\frac{1}{p}-\frac{1}{q})}  t^{-\frac{k\ell}{\alpha}+\frac{1}{\alpha}(\frac{1}{q}-\frac{1}{p})}\|f\|_{L^p},
\end{align}
\begin{align}\label{eq:fracdecay2}
\|[\Lambda^{k\ell} \partial_xG_\alpha(\cdot,t)]\ast f\|_{L^q}\leq C^{k}_{\ell,\alpha}k^{\frac{k\ell+1}{\alpha}+\frac{1}{\alpha}(\frac{1}{p}-\frac{1}{q})}  t^{-\frac{k\ell+1}{\alpha}+\frac{1}{\alpha}(\frac{1}{q}-\frac{1}{p})}\|f\|_{L^p},
\end{align}
and
\begin{align}\label{eq:key}
\lim_{t\to0}t^{\frac{1}{\alpha}(\frac{1}{p}-\frac{1}{q})}\|G_\alpha(\cdot,t)\ast f\|_{L^q}=0,\quad \forall q>p.
\end{align}

\end{lemma}

\begin{proof}
For $\ell\geq 0$, we claim that
\begin{align}\label{eq:LambdaGalpha}
\Lambda^{k\ell} G_\alpha(x,t)=t^{-\frac{k\ell+1}{\alpha}}(\Lambda^{k\ell}  G_\alpha)(t^{-\frac{1}{\alpha}}x,1),\quad \Lambda^{k\ell}\partial_x G_\alpha(x,t)=t^{-\frac{k\ell+2}{\alpha}}(\Lambda^{k\ell}\partial_x  G_\alpha)(t^{-\frac{1}{\alpha}}x,1).
\end{align}
Actually, we have
\[
[\mathcal{F}(\Lambda^{k\ell} G_\alpha)(\xi,t)=|\xi|^{k\ell}e^{-\nu t|\xi|^{\alpha}},
\]
and by changing of variable $\eta=t^{\frac{1}{\alpha}}\xi$ we obtain
\begin{equation*}
\begin{aligned}
\Lambda^{k\ell} G_\alpha(x,t)=&\mathcal{F}^{-1}[\mathcal{F}(\Lambda^{k\ell} G_\alpha)](x,t)=\int_{\mathbb{R}}e^{i\xi x}|\xi|^{k\ell}e^{-\nu t|\xi|^{\alpha}}\di \xi\\
=&t^{-\frac{k\ell+1}{\alpha}}\int_{\mathbb{R}}e^{i t^{-\frac{1}{\alpha}}x\cdot \eta}|\eta|^{k\ell}e^{-\nu \eta}\di \eta=t^{-\frac{k\ell+1}{\alpha}}(\Lambda^{k\ell}  G_\alpha)(t^{-\frac{1}{\alpha}}x,1).
\end{aligned}
\end{equation*}
The proof for the second equality in \eqref{eq:LambdaGalpha} is the same.
Hence, for $1\leq q\leq \infty$
\begin{equation}\label{eq:kell}
\begin{aligned}
\|\Lambda^{k\ell} G_\alpha(\cdot,t)\|_{L^q}=t^{-\frac{k\ell+1}{\alpha}+\frac{1}{\alpha q}}\|\Lambda^{k\ell} G_\alpha(\cdot,1)\|_{L^q},
\end{aligned}
\end{equation}
and
\begin{equation}\label{eq:kell2}
\begin{aligned}
\|\Lambda^{k\ell}\partial_x G_\alpha(\cdot,t)\|_{L^q}=t^{-\frac{k\ell+2}{\alpha}+\frac{1}{\alpha q}}\|\Lambda^{k\ell} \partial_x G_\alpha(\cdot,1)\|_{L^q},
\end{aligned}
\end{equation}
which implies
\begin{align}
\left\|\Lambda^{\ell} G_\alpha\left(\cdot,\frac{1}{k}\right)\right\|_{L^p}=\left(\frac{1}{k}\right)^{-\frac{\ell+1}{\alpha}+\frac{1}{\alpha p} }\|\Lambda^{\ell} G_\alpha(\cdot,1)\|_{L^p}\leq C_{\ell,\alpha}k^{\frac{\ell+1}{\alpha}-\frac{1}{\alpha p} },
\end{align}
and
\begin{align}
\left\|\Lambda^\ell\partial_x G_\alpha\left(\cdot,\frac{1}{k}\right)\right\|_{L^p}=\left(\frac{1}{k}\right)^{-\frac{\ell+2}{\alpha}+\frac{1}{\alpha p}}\|\Lambda^\ell  \partial_x G_\alpha(\cdot,1)\|_{L^p}\leq C_{\ell,\alpha}k^{\frac{\ell+2}{\alpha}-\frac{1}{\alpha p} }.
\end{align}
Due to 
\[
\Lambda^{k\ell} G_\alpha(\cdot,1)=\Lambda^{\ell} G_\alpha\left(\cdot,\frac{1}{k}\right)\ast\Lambda^{\ell} G_\alpha\left(\cdot,\frac{1}{k}\right)\ast\cdots\ast\Lambda^{\ell} G_\alpha\left(\cdot,\frac{1}{k}\right)
\]
and
\[
\Lambda^{k\ell}\partial_x G_\alpha(\cdot,1)=\Lambda^{\ell} G_\alpha\left(\cdot,\frac{1}{k}\right)\ast\cdots\ast\Lambda^{\ell}  G_\alpha\left(\cdot,\frac{1}{k}\right)\ast \Lambda^\ell \partial_xG_\alpha\left(\cdot,\frac{1}{k}\right),
\]
by Young's convolution inequality, we obtain \eqref{eq:fracdecay0} and \eqref{eq:fracdecay00}. Combining \eqref{eq:fracdecay0} and Young's inequality for convolution, we  obtain  \eqref{eq:fracdecay1} and \eqref{eq:fracdecay2}.

The proof of \eqref{eq:key} follows by a similar way as \cite[Lemma 2.1]{carrillo2008asymptotic}. 
Assume $f_n\in L^p(\mathbb{R})\cap L^q(\mathbb{R})$ such that $f_n\to f$ in $L^p(\mathbb{R})$. From \eqref{eq:fracdecay1} with $\ell=0$, we have
\[
\|G_\alpha(\cdot,t)\ast f_n\|_{L^q}\leq C\|f_n\|_{L^q},\quad \forall n\in\mathbb{N}.
\]
Hence, for $q>p$ we obtain
\[
\lim_{t\to0}t^{\frac{1}{\alpha}(\frac{1}{p}-\frac{1}{q})}\|G_\alpha(\cdot,t)\ast f_n\|_{L^q}=0,\quad \forall n\in\mathbb{N}.
\]
According to \eqref{eq:fracdecay1} with $\ell=0$, we also have
\begin{align}\label{eq:positive}
t^{\frac{1}{\alpha}(\frac{1}{p}-\frac{1}{q})}\|G_\alpha(\cdot,t)\ast f_n-G_\alpha(\cdot,t)\ast f\|_{L^q}\leq C\|f_n-f\|_{L^p},
\end{align}
which converges to zero independent of $t>0$. This implies \eqref{eq:key}.
\end{proof}

\subsection{Local existence and uniqueness of mild solutions}\label{sec:local}
Next we are going to prove local existence and uniqueness of mild solutions to \eqref{eq:Hilbert} of the form \eqref{eq:mildsolution} in space $X_T$ defined by \eqref{eq:space1}. For $1<\alpha\leq 2$, assume the initial data $\rho_0\in L^{\frac{1}{\alpha-1}}(\mathbb{R})$. Take $q=\frac{2}{\alpha-1}$ in \eqref{eq:key} and we have
\begin{align}\label{eq:key2}
\lim_{t\to0}t^{\frac{\alpha -1}{2\alpha}}\|G_\alpha(\cdot,t)\ast \rho_0\|_{\frac{2}{\alpha-1}}=0.
\end{align}
Define the operator $S$:
\[
(S\rho)(x,t):=G_\alpha(\cdot,t)\ast\rho_0-\int_0^t \partial_xG_\alpha(\cdot,t-s)\ast(\rho(s) H\rho(s)) \di s.
\]
We have the following theorem: 
\begin{theorem}\label{thm:contraction}
Let $0\leq \rho_0\in L^{\frac{1}{\alpha-1}}(\mathbb{R})$. For $a>0$ small enough and $T>0$ satisfying 
\begin{align}\label{eq:important}
\sup_{0<t\leq T}t^{\frac{\alpha-1}{2\alpha}}\|G_\alpha(\cdot,t)\ast \rho_0\|_{L^{\frac{2}{\alpha-1}}}\leq a,
\end{align}
there exists a unique mild solution $\rho$ to \eqref{eq:Hilbert} in the following subset of $X_T$:
\begin{align*}
X_{T}^a:=\left\{f\in X_{T}:~~\sup_{0< t\leq T}\|f(t)\|_{L^{\frac{1}{\alpha-1}}}\leq 2\|\rho_0\|_{L^\frac{1}{\alpha-1}},~~\sup_{0<t\leq T}t^{\frac{\alpha-1}{2\alpha}}\|f(t)\|_{L^{\frac{2}{\alpha-1}}}\leq 2a\right\}.
\end{align*}
Moreover, we have $\rho\in C([0,T];L^{\frac{1}{\alpha-1}}(\mathbb{R}))$ and $\rho(x,0)=\rho_0(x),~~x\in\mathbb{R}.$
\end{theorem}
There are different versions of the local existence of mild solutions for different equations, and the main ideas are similar; see, e.g. \cite{biler1995cauchy,biler2010blowup,carrillo2008asymptotic}. Since some estimates in the proof of this theorem are useful in the rest of this paper, we are going to provide a complete proof here. 
\begin{proof}
To show the existence, we only need to prove that $S:X_T^a\to X_T^a$ is a contraction mapping for $a$ small enough.

\textbf{Step 1}: Assume $\rho\in X_{T}^a$ and we are going to show that $S\rho\in X_{T}^a$ for $a$ small enough. 

\emph{Estimate of $\|S\rho(t)\|_{L^\frac{1}{\alpha-1}}$:}
From  \eqref{eq:fracdecay2} with $\ell=0$ and $q=p=\frac{1}{\alpha-1}$, by the M. Riesz Theorem for $L^p$ ($1<p<\infty$) boundedness of Hilbert transform we obtain
\begin{align*}
\|\partial_xG_\alpha(\cdot,t-s)\ast (\rho(s)H\rho(s))\|_{L^{\frac{1}{\alpha-1}}}&\leq C(t-s)^{-\frac{1}{\alpha}}\|\rho(s)H\rho(s)\|_{L^{\frac{1}{\alpha-1}}}\\
&\leq C(t-s)^{-\frac{1}{\alpha}}\|\rho(s)\|^2_{L^{\frac{2}{\alpha-1}}}\\
&\leq C(t-s)^{-\frac{1}{\alpha}}s^{-\frac{\alpha-1}{\alpha}}\sup_{0<t\leq T}s^{\frac{\alpha-1}{\alpha}}\|\rho(s)\|^2_{L^{\frac{2}{\alpha-1}}},
\end{align*}
which implies
\begin{multline}\label{eq:rho2}
\int_0^t \|\partial_xG_\alpha(\cdot,t-s)\ast (\rho(s)H\rho(s))\|_{L^{\frac{1}{\alpha-1}}}\di s\leq C a^2\int_0^t(t-s)^{-\frac{1}{\alpha}}s^{-\frac{\alpha-1}{\alpha}}\di s\\
=C\mathcal{B}\left(\frac{1}{\alpha},\frac{\alpha-1}{\alpha}\right)a^2,
\end{multline}
where $\mathcal{B}(\frac{1}{\alpha},\frac{\alpha-1}{\alpha})$ is a Beta function $\mathcal{B}(\alpha,\beta)=\int_0^1s^{\alpha-1}(1-s)^{\beta-1}\di s$ with $\alpha=\frac{1}{\alpha}$ and $\beta=\frac{\alpha-1}{\alpha}$. For $a$ small enough, we have
\[
\|S\rho(t)\|_{L^{\frac{1}{\alpha-1}}}\leq \|G_\alpha(\cdot,t)\ast\rho_0\|_{L^{\frac{1}{\alpha-1}}}+C\mathcal{B}\left(\frac{1}{\alpha},\frac{\alpha-1}{\alpha}\right)a^2<2\|\rho_0\|_{\frac{1}{\alpha-1}}, ~~0\leq t\leq T.
\]

\emph{Estimate of $\sup_{0<t\leq T}t^{\frac{\alpha-1}{2\alpha}}\|S\rho(t)\|_{L^{\frac{2}{\alpha-1}}}$:}
From  \eqref{eq:fracdecay2} with $\ell=0$ and  $q=\frac{2}{\alpha-1},~~p=\frac{1}{\alpha-1}$, we have
\begin{align*}
\|\partial_xG_\alpha(\cdot,t-s)\ast (\rho(s)H\rho(s))\|_{L^{\frac{2}{\alpha-1}}}&\leq C(t-s)^{-\frac{\alpha+1}{2\alpha}}\|\rho(s)H\rho(s)\|_{L^{\frac{1}{\alpha-1}}}\\
&\leq C(t-s)^{-\frac{\alpha+1}{2\alpha}}s^{-\frac{\alpha-1}{\alpha}}\sup_{0<t\leq T}s^{\frac{\alpha-1}{\alpha}}\|\rho(s)\|^2_{L^{\frac{2}{\alpha-1}}},
\end{align*}
which implies
\begin{multline}
\sup_{0<t\leq T}t^{\frac{\alpha-1}{2\alpha}}\int_0^t \|\partial_xG_\alpha(\cdot,t-s)\ast  (\rho(s)H\rho(s))\|_{L^{\frac{2}{\alpha-1}}}\di s \\
\leq Ca^2 \cdot \sup_{0<t\leq T}t^{\frac{\alpha-1}{2\alpha}}\int_0^t (t-s)^{-\frac{\alpha+1}{2\alpha}}s^{-\frac{\alpha-1}{\alpha}} \di s \leq C\mathcal{B}\left(\frac{1}{\alpha},\frac{\alpha-1}{2\alpha}\right)a^2.
\end{multline}
Hence, there exists $a>0$ small enough such that
\begin{align}\label{eq:in1}
\sup_{0<t\leq T}t^{\frac{\alpha-1}{2\alpha}}\|S\rho(t)\|_{L^{\frac{2}{\alpha-1}}}\leq \sup_{0<t\leq T}t^{\frac{\alpha-1}{2\alpha}}\|G_\alpha(\cdot,t)\ast\rho_0\|_{L^{\frac{2}{\alpha-1}}}+C\mathcal{B}\left(\frac{1}{\alpha},\frac{\alpha-1}{2\alpha}\right)a^2<2a.
\end{align}

Since the proof of the continuity of $S\rho$ with respect to time $t\in(0,T]$ in $L^{\frac{1}{\alpha-1}}(\mathbb{R})$ is routine and tedious, we put it in appendix  \ref{apd:continuity}.

\textbf{Step 2.} We are going to show that $S$ is a contraction mapping. 

Consider $\rho_1,~\rho_2\in X_T^a$.
We have
\begin{equation}\label{eq:contraction}
\begin{aligned}
\|S\rho_1-&S\rho_2\|_{X_T}\leq \left\|\int_0^t\partial_xG_\alpha(\cdot,t-s)\ast \left[(\rho_1(s)-\rho_2(s))H\rho_1(s)\right]\di s\right\|_{X_T}\\
&+ \left\|\int_0^t\partial_xG_\alpha(\cdot,t-s)\ast \left[\rho_2(s)H (\rho_1(s)-\rho_2(s))\right]\di s\right\|_{X_T}
\end{aligned}
\end{equation}
Similarly to Step 1, we have the following estimates for the first term on the right hand side of \eqref{eq:contraction}:
\begin{align*}
&\sup_{0\leq t\leq T}\left\|\int_0^t\partial_xG_\alpha(\cdot,t-s)\ast \left[(\rho_1(s)-\rho_2(s))H\rho_1(s)\right]\di s\right\|_{L^{\frac{1}{\alpha-1}}}\\
\leq &C \mathcal{B}\left(\frac{1}{\alpha},\frac{\alpha-1}{\alpha}\right)a \|\rho_1-\rho_2\|_{X_T},
\end{align*}
and
\begin{align*}
&\sup_{0\leq t\leq T}t^{\frac{\alpha-1}{2\alpha}}\left\|\int_0^t\partial_xG_\alpha(\cdot,t-s)\ast \left[(\rho_1(s)-\rho_2(s))H\rho_1(s)\right]\di s\right\|_{L^{\frac{2}{\alpha-1}}}\\
\leq &C \mathcal{B}\left(\frac{1}{\alpha},\frac{\alpha-1}{2\alpha}\right)a\left\|\rho_1-\rho_2\right\|_{X_T}.
\end{align*}
We have similar estimate for the second term in the right hand of \eqref{eq:contraction}. Hence
\begin{align}\label{eq:cont}
\left\|S\rho_1 -S\rho_2 \right\|_{X_T}\leq Ca\left\|\rho_1-\rho_2\right\|_{X_T}.
\end{align}
And $S$ is a contraction mapping for small enough $a>0$.

\textbf{Step 3.} In this step, we are going to show the time continuity at $t=0$. Since 
\[
\lim_{t\to0}\|G_\alpha(\cdot,t)\ast\rho_0-\rho_0\|_{L^{\frac{1}{\alpha-1}}}=0,
\]
hence we only need to show 
\begin{align}\label{eq:last}
\lim_{t\to0}\left\|\int_0^t\partial_xG_\alpha(\cdot,t-s)\ast(\rho(s)H\rho(s))\di s\right\|_{L^{\frac{1}{\alpha-1}}}=0.
\end{align}
Let $\rho$ be the solution constructed by Step 1 and Step 2. Consider another two positive numbers $\tilde{a}<a$ and $\tilde{T}<T$ such that \eqref{eq:important} holds for $\tilde{a}$ and $\tilde{T}$. We could obtain another solution $\tilde{\rho}$ in the corresponding space $X_{\tilde{T}}^{\tilde{a}}$. Similarly to \eqref{eq:cont}, we obtain
\[
\|\tilde{\rho}-\rho\|_{X_{\tilde{T}}}\leq Ca\|\tilde{\rho}-\rho\|_{X_{\tilde{T}}},
\]
which implies $\tilde{\rho}(t)=\rho(t)$ for $0<t\leq \tilde{T}$.
Due to \eqref{eq:key2}, as $\tilde{T}\to0$, we could choose $\tilde{a}\to0$
Combining \eqref{eq:rho2} for $\tilde{a}\to0$, we have \eqref{eq:last}.

\end{proof}

\subsection{Regularity, nonnegativity, and analyticity}\label{sec:regularity}
In this section, we improve the regularity of mild solutions step by step and obtain the nonnegativity. Then, we show the spatial analyticity of solutions.

 For the regularity of mild solutions, the strategy is as follows. We first show the hypercontrativity estimate that the mild solutions belong to $L^q$ spaces for any $\frac{1}{\alpha-1}\leq q\leq \infty$. Then, we estimate the derivative of mild solutions. The time decay property can only ensure the improvement of fractional step $0<\ell<\alpha-1$ for the derivative. Step by step, we could improve the regularity of solutions to any order we want. For the nonegativity, we follow the same idea as in \cite[Lemma 2.7]{li2011one}. However, comparing with  \cite[Lemma 2.7]{li2011one} where the initial data $\rho_0\in L^{\frac{1}{\alpha-1}}(\mathbb{R})\cap L^p(\mathbb{R})$ for some $\frac{1}{\alpha-1}<p<\infty$, we only need the initial data $\rho_0\in L^{\frac{1}{\alpha-1}}(\mathbb{R})$. 
We have the following theorem:
\begin{theorem}[Regularity and nonnegativity]\label{thm:regularity}
Let $0\leq \rho_0\in  L^{\frac{1}{\alpha-1}}(\mathbb{R})$. Then, the mild solution $\rho$ obtained by Theorem \ref{thm:contraction} is a strong solution for $t>0$ belonging to $C^\infty((0,T];H^{\theta,q}(\mathbb{R}))$ for any $\theta>0$ and $\frac{1}{\alpha-1}\leq q\leq \infty$. The following time decay estimates for derivatives  hold:
\begin{align}\label{eq:Wqestimate}
\|\rho(t)\|_{\dot{H}^{\theta,q}(\mathbb{R})}\leq Ct^{-\frac{\theta}{\alpha}-1+\frac{1}{\alpha}(1+\frac{1}{q})},\quad  \frac{1}{\alpha-1}\leq q\leq \infty,~~0<t\leq T,
\end{align}
and 
\begin{align}\label{eq:decaypower}
\|\partial_x^n\rho(t)\|_{L^\infty}\leq Ct^{-\frac{n}{\alpha}-1+\frac{1}{\alpha}(1+\frac{1}{q})},\quad\forall n\in\mathbb{N},~~0<t\leq T.
\end{align}
Moreover, $\rho(x,t)\geq0$ for any $t\in[0,T]$, $x\in\mathbb{R}$.
\end{theorem}
\begin{proof}
Denote
\begin{align}\label{eq:rho12}
\rho(x,t)=\rho_1(x,t)-\rho_2(x,t):=G_\alpha(\cdot,t)\ast\rho_0-\int_0^t \partial_xG_\alpha(\cdot,t-s)\ast (\rho(s) H\rho(s))\di s.
\end{align}
The first term $\rho_1(x,t)=G_\alpha(\cdot,t)\ast \rho_0$ is the solution to the fractional heat equation with initial date $\rho_0$. Due to instantaneous regularization of the fractional heat equation, we have $\rho_1\in C^\infty(\mathbb{R}\times(0,\infty))$. From  \eqref{eq:fracdecay1} with $p=\frac{1}{\alpha-1}$, we obtain for $\frac{1}{\alpha-1}\leq q\leq \infty,$ $0< t\leq T$,
\begin{equation}\label{eq:Lq1}
\begin{aligned}
\|\rho_1(t)\|_{\dot{H}^{k\ell,q}}=\|\Lambda^{k\ell}G_\alpha(\cdot,t)\ast\rho_0\|_{L^q}\leq C_{\ell,\alpha}^kk^{1+\frac{k\ell}{\alpha}-\frac{1}{\alpha}(1+\frac{1}{q})} t^{-1-\frac{k\ell}{\alpha}+\frac{1}{\alpha}(1+\frac{1}{q})}.
\end{aligned}
\end{equation}
Next, we separate the proof into several steps.

\textbf{Step 1.} In this step, we are going to prove 
\begin{align}\label{eq:Lqestimate}
\|\rho(t)\|_{L^q}\leq C_{0,\alpha}t^{-1+\frac{1}{\alpha}(1+\frac{1}{q})},\quad  \frac{1}{\alpha-1}\leq q\leq \infty,~~0<t\leq T.
\end{align}
Because of \eqref{eq:Lq1} (for $\ell=0$), we only need to show that $\rho_2$ satisfies \eqref{eq:Lqestimate}.
From  \eqref{eq:fracdecay2} with $\ell=0$ and $\frac{1}{\alpha-1}=p\leq q< \infty$, we have
\begin{equation}\label{eq:Lq2}
\begin{aligned}
\|\rho_2(t)\|_{L^q}\leq &C\int_0^t (t-s)^{-\frac{1}{\alpha}(\alpha-\frac{1}{q})}\|\rho(s) H\rho(s)\|_{L^{\frac{1}{\alpha-1}}}\di s \\
\leq &C\int_0^t (t-s)^{-1+\frac{1}{\alpha q}}s^{-\frac{\alpha-1}{\alpha}}\di s\cdot\sup_{0<t\leq T}s^{\frac{\alpha-1}{\alpha}}\|\rho(s)\|^2_{L^{\frac{2}{\alpha-1}}} \\
=&Ct^{-1+\frac{1}{\alpha}(1+\frac{1}{q})}.
\end{aligned}
\end{equation}
Hence, \eqref{eq:Lqestimate} holds for $ \frac{1}{\alpha-1}\leq q<\infty$. 

From  \eqref{eq:fracdecay2} with $\ell=0$, $p=\frac{2}{\alpha-1}$ and $q=\infty,$ we have
\begin{equation*}
\begin{aligned}
\|\rho_2(t)\|_{L^\infty}\leq &C\int_0^t (t-s)^{-\frac{\alpha+1}{2\alpha}}\|\rho(s) H\rho(s)\|_{L^{\frac{2}{\alpha-1}}}\di s \\
\leq &C\int_0^t (t-s)^{-\frac{\alpha+1}{2\alpha}}s^{-\frac{3(\alpha-1)}{2\alpha}}\di s\cdot \sup_{0<t\leq T}s^{\frac{3(\alpha-1)}{2\alpha}}\|\rho(s)\|^2_{L^{\frac{4}{\alpha-1}}}.
\end{aligned}
\end{equation*}
Because $-1<-\frac{\alpha+1}{2\alpha}<0$ and $-1<-\frac{3(\alpha-1)}{2\alpha}<0$, from \eqref{eq:Lq2} we obtain
\begin{equation}\label{eq:Lq3}
\begin{aligned}
\|\rho_2(t)\|_{L^\infty}\leq Ct^{-1+\frac{1}{\alpha}},\quad 0<t\leq T.
\end{aligned}
\end{equation}
Combining \eqref{eq:Lq1}, \eqref{eq:Lq2} and \eqref{eq:Lq3} gives \eqref{eq:Lqestimate}.

\textbf{Step 2.}
In this step, we are going to prove $\rho(t)\in H^{\ell,q}(\mathbb{R})$ for any $0<\ell<\alpha-1$, $\frac{1}{\alpha-1}\leq q\leq \infty$, and
\begin{align}\label{eq:step2}
\|\rho(t)\|_{\dot{H}^{\ell,q}} \leq Ct^{-1-\frac{\ell}{\alpha}+\frac{1}{\alpha}(1+\frac{1}{q})},\quad \frac{1}{\alpha-1}\leq q\leq \infty, \quad 0< t\leq T.
\end{align}
Because of \eqref{eq:Lq1}, we only need to show \eqref{eq:step2} for $\rho_2$.

From \eqref{eq:fracdecay2} with $p=q$, we have
\begin{equation*}
\begin{aligned}
\|\rho_2(t)\|_{\dot{H}^{\ell,q}} =& \|\Lambda^\ell\rho_2(t)\|_{L^q}\leq \int_0^t \|\Lambda^\ell\partial_xG_\alpha(\cdot,t-s)\ast (\rho(s) H\rho(s))\|_{L^q}\di s\\
\leq &C\int_0^t  (t-s)^{-\frac{\ell+1}{\alpha}}\|\rho(s) H\rho(s)\|_{L^q} \di s\\
\leq &C\int_0^t  (t-s)^{-\frac{\ell+1}{\alpha}}s^{-2+\frac{1}{\alpha}(2+\frac{1}{q})} \di s\cdot \sup_{0<s\leq T}s^{2-\frac{1}{\alpha}(2+\frac{1}{q})}\|\rho(s)\|_{L^{2q}}^2.
\end{aligned}
\end{equation*}
For any $\frac{1}{\alpha-1}\leq q<\infty$, we have $-2+\frac{1}{\alpha}(2+\frac{1}{q})>-1.$ Hence, for $0<\ell<\alpha-1$, we have
\begin{equation}\label{eq:rho2wq}
\begin{aligned}
\|\rho_2(t)\|_{\dot{H}^{\ell,q}} \leq C\int_0^t  (t-s)^{-\frac{\ell+1}{\alpha}}s^{-2+\frac{1}{\alpha}(2+\frac{1}{q})} \di s\leq Ct^{-1-\frac{\ell}{\alpha}+\frac{1}{\alpha}(1+\frac{1}{q})}.
\end{aligned}
\end{equation}
From \eqref{eq:fracdecay2} with $q=\infty$, we have
\begin{equation*}
\begin{aligned}
\|\rho_2(t)\|_{\dot{H}^{\ell,\infty}} =& \|\Lambda^\ell\rho_2(t)\|_{L^\infty}\leq \int_0^t \|\Lambda^\ell\partial_xG_\alpha(\cdot,t-s)\ast (\rho(s) H\rho(s))\|_{L^\infty}\di s\\
\leq &C\int_0^t  (t-s)^{-\frac{\ell+1}{\alpha}-\frac{1}{p\alpha}}\|\rho(s) H\rho(s)\|_{L^p} \di s\\
\leq &C\int_0^t  (t-s)^{-\frac{\ell+1}{\alpha}-\frac{1}{p\alpha}}s^{-2+\frac{1}{\alpha}(2+\frac{1}{p})} \di s\cdot \sup_{0<s\leq T}s^{2-\frac{1}{\alpha}(2+\frac{1}{p})}\|\rho(s)\|_{L^{2p}}^2.
\end{aligned}
\end{equation*}
Hence, for $0<\ell<\alpha-1$ and $p<\infty$ big enough, we have $-\frac{\ell+1}{\alpha}-\frac{1}{p\alpha}>-1$ and \eqref{eq:step2} holds for $q=\infty$.

\textbf{Step 3.} 
In this step, we are going to prove  that if $\rho(t)\in H^{\beta,q}(\mathbb{R})$ for any $\beta>0$ and $\frac{1}{\alpha-1}\leq q\leq \infty$ satisfying
\begin{align}\label{eq:step31}
\|\rho(t)\|_{\dot{H}^{\beta,q}} \leq Ct^{-1-\frac{\beta}{\alpha}+\frac{1}{\alpha}(1+\frac{1}{q})},
\end{align}
then we have $\rho(t)\in H^{\beta+\ell,q}(\mathbb{R})$ and
\begin{align}\label{eq:step3}
\|\rho(t)\|_{\dot{H}^{\beta+\ell,q}} \leq Ct^{-1-\frac{\beta+\ell}{\alpha}+\frac{1}{\alpha}(1+\frac{1}{q})},\quad \frac{1}{\alpha-1}\leq q\leq \infty, \quad 0< t\leq T.
\end{align}

Since \eqref{eq:Lq1}, we only need to show that $\rho_2(x,t)$ satisfies \eqref{eq:step3}. 
Notice that $\Lambda^\beta(H\rho(s))=H(\Lambda^\beta\rho(s))$ and hence for $\frac{1}{\alpha-1}\leq q<\infty,$
\[
\|\Lambda^\beta(H\rho(s))\|_{L^{q}}=\|H(\Lambda^\beta\rho(s))\|_{L^{q}}\leq \|\Lambda^\beta\rho(s)\|_{L^q}<\infty.
\]
This implies $H\rho\in H^{\beta,q}(\mathbb{R})$ for $\frac{1}{\alpha-1}\leq q<\infty$.  By the Sobolev embedding for $q$ big enough, we know $H\rho(s)\in L^\infty(\mathbb{R})$. Therefore,
\[
\rho(s),~H\rho(s)\in H^{\beta,q}(\mathbb{R})\cap L^\infty(\mathbb{R}).
\]
From fractional Leibniz inequality (see, for instance, \cite{kato1988commutator}), we have
\begin{equation}\label{eq:KatoPonce}
\begin{aligned}
\|\Lambda^\beta(\rho(s)H\rho(s))\|_{L^q}\leq C\Big(\|\Lambda^\beta\rho(s)\|_{L^{q_1}}\|H\rho_2(s)\|_{L^{q_2}}+\|\Lambda^\beta H\rho_2(s)\|_{L^{q_3}}\|\rho_2(s)\|_{L^{q_4}}\Big)
\end{aligned}
\end{equation}
for 
\[
\frac{1}{q}=\frac{1}{q_1}+\frac{1}{q_2}=\frac{1}{q_3}+\frac{1}{q_4},\quad q<q_i,\quad i=1,\cdots,4.
\]
Combining \eqref{eq:Lqestimate}, \eqref{eq:step31} and \eqref{eq:KatoPonce} gives
\begin{equation}\label{eq:productdecay}
\|\rho(s)H\rho(s)\|_{\dot{H}^{\beta,q}}=\|\Lambda^\beta(\rho(s)H\rho(s))\|_{L^q}\leq Ct^{-2-\frac{\beta}{\alpha}+\frac{1}{\alpha}(2+\frac{1}{q})}.
\end{equation}
From \eqref{eq:fracdecay1},  \eqref{eq:fracdecay2} and \eqref{eq:productdecay}, the following holds for $\frac{1}{\alpha-1}\leq q\leq \infty$
\begin{equation*}
\begin{aligned}
& \|\Lambda^{\beta+\ell}\rho_2(t)\|_{L^q}\\
\leq& \int_{\delta}^t \|\Lambda^\ell\partial_xG_\alpha(\cdot,t-s)\ast [\Lambda^\beta(\rho(s) H\rho(s))]\|_{L^q}\di s + \int_0^{\delta} \|\Lambda^{\beta+\ell}\partial_xG_\alpha(\cdot,t-s)\ast (\rho(s) H\rho(s))\|_{L^q}\di s\\
\leq &C\int_{\delta}^t  (t-s)^{-\frac{\ell+1}{\alpha}+\frac{1}{\alpha}(\frac{1}{q}-\frac{1}{p})}\|\Lambda^\beta(\rho(s) H\rho(s))\|_{L^p} \di s +C\int^{\delta}_0  (t-s)^{-\frac{\beta+\ell+1}{\alpha}+\frac{1}{\alpha}(\frac{1}{q}-\frac{1}{p})}\|\rho(s) H\rho(s)\|_{L^p} \di s\\
\leq &C\int_{\delta}^t  (t-s)^{-\frac{\ell+1}{\alpha}+\frac{1}{\alpha}(\frac{1}{q}-\frac{1}{p})}s^{-2-\frac{\beta}{\alpha}+\frac{1}{\alpha}(2+\frac{1}{p})} \di s +C\int_0^{\delta}  (t-s)^{-\frac{\beta+\ell+1}{\alpha}+\frac{1}{\alpha}(\frac{1}{q}-\frac{1}{p})}s^{-2+\frac{1}{\alpha}(2+\frac{1}{p})} \di s\\
\leq & C(t-\delta)^{1-\frac{\ell+1}{\alpha}+\frac{1}{\alpha}(\frac{1}{q}-\frac{1}{p})}\delta^{-2-\frac{\beta}{\alpha}+\frac{1}{\alpha}(2+\frac{1}{p})}+ C(t-\delta)^{-\frac{\beta+\ell+1}{\alpha}+\frac{1}{\alpha}(\frac{1}{q}-\frac{1}{p})}\delta^{-1+\frac{1}{\alpha}(2+\frac{1}{p})}.
\end{aligned}
\end{equation*}
Choose $\delta=\frac{t}{2}$ and we obtain \eqref{eq:step3} for $\rho_2(x,t)$.

\textbf{Step 4.} Notice that the Bessel potential space $H^{\theta,\infty}(\mathbb{R})$ is not the same as Sobolev space $W^{\theta,\infty}$. Since $H^{\theta,q}(\mathbb{R})=W^{\theta,q}(\mathbb{R})$ for any $\theta>0$ and $1<q<\infty$, from Step 3 we see that the estimate for $W^{\theta,\infty}(\mathbb{R})$ is the same as the estimate of $H^{\theta,\infty}(\mathbb{R})$. So we omit the proof of \eqref{eq:decaypower}.

\textbf{Step 5.} In this step, we are going to show the regularity of time for the mild solution $\rho$. First, let us prove $\rho\in C((0,T]; H^{\theta,q}(\mathbb{R}))$ for any $\theta>0$ and $q\geq\frac{1}{\alpha-1}$. For any $t>s>0$, we have
\[
\rho(t)=G_\alpha(\cdot,t-s)\ast\rho(s)-\int_s^t\partial_xG_\alpha(\cdot,t-\tau)\ast(\rho(\tau)H\rho(\tau))\di\tau.
\]
Therefore,
\begin{multline}\label{eq:continu}
\|\rho(t)-\rho(s)\|_{H^{\theta,q}}\leq \|G_\alpha(\cdot,t-s)\ast\rho(s)-\rho(s)\|_{H^{\theta,q}}\\
+\int_s^t\|G_\alpha(\cdot,t-\tau)\ast \partial_x (\rho(\tau)H\rho(\tau))\|_{H^{\theta,q}}\di\tau.
\end{multline}
The first term in the right hand side of \eqref{eq:continu} goes to zero as $|s-t|\to0$, because the solution of fractional heat equation is continuous at the initial data $\rho(s)$ in $H^{\theta,q}(\mathbb{R})$. For the second term in \eqref{eq:continu}, due to \eqref{eq:fracdecay1} for $\ell=0$, we have the following estimate:
\[
\|G_\alpha(\cdot,t-\tau)\ast\partial_x (\rho(\tau)H\rho(\tau))\|_{H^{\theta,q}}\leq C\|\rho(\tau)H\rho(\tau)\|_{H^{\theta+1,q}}.
\]
From Step 3, we know that $\|\rho(\tau)\|_{H^{\theta+1,q}}$ is uniformly bounded for $\tau\in(s, t)$. Therefore, the second term in the right hand side of \eqref{eq:continu} also goes to zero as $|t-s|\to0$.

Next, we improve the time regularity. Choose an arbitrary $t_0\in(0,T)$ and set the new initial date
\[
\tilde{\rho}_0:=\rho(t_0).
\]
With this new initial date, we have a mild solution
\[
\tilde{\rho}(t)=\rho(t+t_0),\quad t\in[0,T-t_0],
\]
which satisfies
\begin{equation}\label{eq:barrho12}
\begin{aligned}
\tilde{\rho}(x,t)=G_\alpha(\cdot,t)\ast\tilde{\rho}_0-\int_0^t  \partial_xG_\alpha(\cdot,t-s)\ast\left(\tilde{\rho}(s)H\tilde{\rho}(s)\right)\di s.
\end{aligned}
\end{equation}
Notice that for $f\in  L^p(\mathbb{R})$, because $G_\alpha(\cdot,t)\ast f$ is the solution of fractional heat equation  $u_t=-\nu\Lambda^\alpha u$ with initial data $f$, the following holds:
\begin{align}\label{eq:heatproperti}
G_\alpha(\cdot,t)\ast f-f=-\int_0^t\nu\Lambda^\alpha G_\alpha(\cdot,s)\ast f\di s.
\end{align}
Combining \eqref{eq:barrho12} and \eqref{eq:heatproperti} yields
\begin{align*}
&\int_0^t\left[-\nu\Lambda^\alpha \tilde{\rho}(\tau)-\partial_x(\tilde{\rho}(\tau)H\tilde{\rho}(\tau))\right]\di \tau\\
=&\int_0^t-\nu\Lambda^\alpha G_\alpha(\cdot,\tau)\ast\tilde{\rho}_0\di \tau -\int_0^t\int_0^\tau -\nu\Lambda^\alpha[\partial_xG_\alpha(\cdot, \tau-s)\ast(\tilde{\rho}(s)H\tilde{\rho}(s))]\di s\di \tau\\
&\qquad -\int_0^t\partial_x(\tilde{\rho}(\tau)H\tilde{\rho}(\tau))\di \tau\\
=&(G_\alpha(\cdot,t)\ast\tilde{\rho}_0-\tilde{\rho}_0)-\int_0^t\int_s^t-\nu\Lambda^\alpha [\partial_xG_\alpha(\cdot,\tau-s)\ast(\tilde{\rho}(s)H\tilde{\rho}(s))]\di \tau\di s\\
&\qquad -\int_0^t\partial_x(\tilde{\rho}(\tau)H\rho(\tau))\di \tau\\
=&G_\alpha(\cdot,t)\ast\tilde{\rho}_0-\int_0^t\partial_x G_\alpha(\cdot,t-s)\ast(\tilde{\rho}(s)H\rho(s))\di s-\tilde{\rho}_0,
\end{align*}
which implies
\begin{align}\label{eq:keyidentity}
\int_0^t[-\nu\Lambda^\alpha\tilde{\rho}(\tau)-\partial_x(\tilde{\rho}(\tau)H\tilde{\rho}(\tau))]\di \tau=\tilde{\rho}(t)-\tilde{\rho}_0.
\end{align}
Because $\partial_x(\tilde{\rho}H\tilde{\rho})\in C([0,T-t_0]; H^{\theta,q}(\mathbb{R})),$ $\forall \theta>0,~~q\geq\frac{1}{\alpha-1}$ we have
\begin{align}
\tilde{\rho}\in C^\infty([0,T-t_0]; H^{\theta,q}(\mathbb{R})), \quad \forall \theta>0.
\end{align}
Since $t_0$ is chosen arbitrarily, the time regularity is obtained.

\textbf{Step 6}. In this step, we are going to show  the nonnegativity of solutions. 

Consider a sequence of smooth positive functions $\rho_{0n}$ such that $\rho_{0n}\to \rho_0$ in $L^{\frac{1}{\alpha-1}}(\mathbb{R})$ as $n\to\infty.$ By \eqref{eq:positive} for $p=\frac{1}{\alpha-1}$ and $q=\frac{2}{\alpha-1}$, we have
\[
\|G_\alpha(\cdot,t)\ast \rho_{0n}-G_\alpha(\cdot,t)\ast \rho_{0}\|_{X_T}\to 0,\quad n\to\infty.
\]
Recall Theorem \ref{thm:contraction}. For $a$ small enough, there exists (a uniform) $T>0$ such that \eqref{eq:important} holds for all $n\in\mathbb{N}$, i.e.,
\[
\sup_{0<t\leq T}t^{\frac{\alpha-1}{2\alpha}}\|G_\alpha(\cdot,t)\ast \rho_{0n}\|_{L^{\frac{2}{\alpha-1}}}\leq a.
\]
Hence, Theorem \ref{thm:contraction} gives a sequence of solutions $\rho_n \in X_T^a$ with initial data $\rho_{0n}$ in a uniform time interval $[0,T]$. Due to \cite[Lemma 2.7]{li2010exploding}, we also have $\rho_n\geq0$. 
Similarly to \eqref{eq:cont}, we obtain
\begin{equation}
\begin{aligned}
\|\rho_n-\rho\|_{X_T}\leq \|G_\alpha(\cdot,t)\ast \rho_{0n}-G_\alpha(\cdot,t)\ast \rho_{0}\|_{X_T}+Ca\|\rho_n-\rho\|_{X_T},
\end{aligned}
\end{equation}
which implies 
\begin{equation}
\begin{aligned}
\|\rho_n-\rho\|_{X_T}\leq \frac{1}{1-Ca}\|G_\alpha(\cdot,t)\ast \rho_{0n}-G_\alpha(\cdot,t)\ast \rho_{0}\|_{X_T}\to 0,\quad n\to\infty.
\end{aligned}
\end{equation}
Because $\rho_{n},~\rho\in C([0,T];L^{\frac{1}{\alpha-1}}(\mathbb{R}))$, we have $\|\rho_{n}(t)-\rho(t)\|_{L^{\frac{1}{\alpha-1}}}\to0$ as $n\to\infty$ for any $0<t<T$. Since $\rho_{n}(t)\geq0$, we obtain $\rho(t)\geq0$  for a.e. $x\in\mathbb{R}$. Since $\rho(t)\in C^\infty(\mathbb{R})$ for $t>0$, we obtain $\rho(x,t)\geq0$ for all $x\in\mathbb{R}$.

This is the end of the proof.
\end{proof}

To prove the spatial analyticity of mild solutions, we need to obtain more explicit estimate for the constant in \eqref{eq:decaypower}. We are going to generalize the method used in \cite{giga2002regularizing,sawada2005analyticity} for the cases of fractional diffusion. First, let us introduce a useful lemma about an estimate for multiplication of sequences, which was proved by C. Kahane \cite[Lemma 2.1]{kahane1969spatial}. The original lemma is for multi-index, and here we only need to use the following one dimensional version for integers.
\begin{lemma}\label{lmm:sequesmulti}
Let $\delta>\frac{1}{2}$. Then there exists a positive constant $\lambda$ depending only on $\delta$ such that
\[
\sum_{0\leq j\leq k}\binom{k}{j}j^{j-\delta}(k-j)^{k-j-\delta}\leq\lambda k^{k-\delta},~~\forall k\in\mathbb{N}.
\]
Here, we use $0^p=1$ for any $p\in\mathbb{R}.$
\end{lemma}
We have the following more explicit estimate for \eqref{eq:decaypower}:
\begin{theorem}[Spatial Analyticity]\label{thm:spatialanaly}
Let $\rho(t)$ be a mild solution given in Theorem \ref{thm:regularity}. Then 
\begin{align}\label{eq:decaypower1}
\|\partial_x^n\rho(t)\|_{L^q(\mathbb{R})}\leq K^{n}n^{n}t^{-\frac{n}{\alpha}-1+\frac{1}{\alpha}(1+\frac{1}{q})},\quad\forall n\in\mathbb{N},~~0<t\leq T
\end{align}
for some constant $K$ independent of $n$ and $\frac{1}{\alpha-1}\leq q\leq \infty$. Consequently, $\rho(\cdot,t)$ is spatially analytic for $0<t\leq T.$
\end{theorem}
\begin{proof}
Let $n \in\mathbb{N}$. Notice that we only need to prove \eqref{eq:decaypower1} for $n$ big enough. We use induction to prove this.
Assume that there exist constants $K$ (to be fixed) and $\delta>\frac{1}{2}$ such that
\begin{align}\label{eq:decaypower2}
\|\partial_x^m\rho(t)\|_{L^q}\leq K^{m-\delta}m^{m-\delta}t^{-\frac{m}{\alpha}-1+\frac{1}{\alpha}(1+\frac{1}{q})},\quad ~~0<t\leq T
\end{align}
holds for any $\frac{1}{\alpha-1}<q<\infty$ and $m<n$. Then, we prove that \eqref{eq:decaypower2} also holds for $m=n$. The cases for $q=\infty$ and $q=\frac{1}{\alpha-1}$ follow easily after we obtain the results for $\frac{1}{\alpha-1}<q<\infty$.

Due to the regularity results Theorem \ref{thm:regularity}, we only need to show \eqref{eq:decaypower2} for $n$ large enough. We have
\begin{equation}\label{eq:main}
\begin{aligned}
&\|\partial_x^n\rho(t)\|_{L^q}\leq \|\partial_x^n\rho_1(t)\|_{L^q}+\int_0^{t}\|\partial_x^{n}\partial_xG_\alpha(\cdot,t-s)\ast(\rho(s)H\rho(s))\|_{L^q}\di s\\
\leq& \|\partial_x^n\rho_1(t)\|_{L^q}+\sum_{m=2}^{n-1}\sum_{j=1}^m\binom{m}{j} \int_{\frac{m-1}{n}t}^{\frac{m}{n}t}\|\partial_x^{n-m}\partial_xG_\alpha(\cdot,t-s)\ast (\partial_x^j\rho(s)\partial_x^{m-j}H\rho(s))\|_{L^q}\di s\\
&\qquad \qquad +\int_{0}^{\frac{1}{n}t}\|\partial_x^{n}\partial_xG_\alpha(\cdot,t-s)\ast ( \rho(s) H\rho(s))\|_{L^q}\di s\\
&\qquad \qquad +  \int_{\frac{n-1}{n}t}^{t}\sum_{1<j<n}\binom{n}{j} \|\partial_xG_\alpha(\cdot,t-s)\ast (\partial_x^j\rho(s)\partial_x^{n-j}H\rho(s))\|_{L^q}\di s\\
&\qquad \qquad + \int_{\frac{n-1}{n}t}^{t} \|\partial_xG_\alpha(\cdot,t-s)\ast [ \rho(s)\partial_x^{n}H\rho(s)+\partial_x^n\rho(s) H\rho(s)]\|_{L^q}\di s\\
=&:A_1+A_2+A_3+A_4+A_5.
\end{aligned}
\end{equation}
Let $\mu:=\frac{n}{\alpha}+1-\frac{1}{\alpha}(1+\frac{1}{q})$.
Next, we are going to estimate $A_i$ respectively, $1\leq i\leq 5$. From \eqref{eq:Lq1}, choose $K$ big enough and we have
\begin{equation*}
A_1=\|\partial_x^n\rho_1(t)\|_{L^q}\leq C_{0,\alpha}^nn^{\mu} t^{-\mu}\leq t^{-\mu}K^{n-2\delta}n^{n-\delta}.
\end{equation*}
For $A_2$, we use \eqref{eq:fracdecay2} for $\frac{1}{\alpha-1}\leq p\leq q$ to obtain
\begin{equation*}
\begin{aligned}
A_2\leq \sum_{m=2}^{n-1}C_{1,\alpha}^{n-m}(n-m)^{\frac{n-m}{\alpha}+a} \int_{\frac{m-1}{n}t}^{\frac{m}{n}t}(t-s)^{-\frac{n-m}{\alpha}-a}\sum_{j=1}^m\binom{m}{j}\|\partial_x^j\rho(s)\partial_x^{m-j}H\rho(s)\|_{L^p}\di s,
\end{aligned}
\end{equation*}
where we choose $p$ close enough to $q$ such that $a=:\frac{1}{\alpha}(1+\frac{1}{p}-\frac{1}{q})<1$. Combining H\"older's inequality, \eqref{eq:decaypower2} and Lemma \ref{lmm:sequesmulti} gives 
\[
\sum_{j=1}^m\binom{m}{j}\|\partial_x^j\rho(s)\partial_x^{m-j}H\rho(s)\|_{L^p}\leq K^{m-2\delta}m^{m-\delta}s^{-\frac{m}{\alpha}-b}
\]
for $0<b:=2-\frac{1}{\alpha}(2+\frac{1}{p})<1$. Therefore, choose $K$ big enough and we have
\begin{equation*}
\begin{aligned}
A_2\leq& \sum_{m=2}^{n-1}C_{1,\alpha}^{n-m}(n-m)^{\frac{n-m}{\alpha}+a}K^{m-2\delta}m^{m-\delta} \int_{\frac{m-1}{n}t}^{\frac{m}{n}t}(t-s)^{-\frac{n-m}{\alpha}-a}s^{-\frac{m}{\alpha}-b}\di s\\
\leq&t^{-\mu} K^{n-2\delta}\sum_{m=2}^{n-1} (n-m)^{\frac{n-m}{\alpha}+a}m^{m-\delta} \int_{\frac{m-1}{n}}^{\frac{m}{n}}(1-\tau)^{-\frac{n-m}{\alpha}-a}\tau^{-\frac{m}{\alpha}-b}\di \tau\\
\leq&t^{-\mu} K^{n-2\delta}\sum_{m=2}^{n-1}\frac{1}{n} (n-m)^{\frac{n-m}{\alpha}+a}m^{m-\delta} \left(\frac{n}{n-m}\right)^{\frac{n-m}{\alpha}+a}\left(\frac{n}{m-1}\right)^{\frac{m}{\alpha}+b}\\
=&t^{-\mu} K^{n-2\delta}\sum_{m=2}^{n-1} n^{\mu} m^{m-\delta}  \left(\frac{1}{m-1}\right)^{\frac{m}{\alpha}+b}.
\end{aligned}
\end{equation*}
We claim that for $n$ big enough we have
\begin{align}\label{eq:claim1}
\sum_{m=2}^{n-1} n^{\mu} m^{m-\delta}  \left(\frac{1}{m-1}\right)^{\frac{m}{\alpha}+b}\leq  n^{n-\delta}.
\end{align}
See the proof of \eqref{eq:claim1} in Appendix \ref{app:proofclaim1}. Hence
\begin{equation*}
A_2\leq t^{-\mu} K^{n-2\delta}n^{n-\delta}.
\end{equation*}
For $A_3$ and $A_4$ in \eqref{eq:main}, choose $K$ big enough and we have
\begin{equation*}
\begin{aligned}
A_3=&\int_{0}^{\frac{1}{n}t}\|\partial_x^{n}\partial_xG_\alpha(\cdot,t-s)\ast ( \rho(s) H\rho(s))\|_{L^q}\di s\leq t^{-\mu} C_{1,\alpha}^n n^{\frac{n}{\alpha}+a}\int_0^{\frac{1}{n}}(1-\tau)^{-\frac{n}{\alpha}-a}\tau^{-b}\di\tau\\
\leq& t^{-\mu} C_{1,\alpha}^n n^{\frac{n}{\alpha}+a}\int_0^{\frac{1}{n}} \frac{1}{1-b} [(1-\tau)^{-\frac{n}{\alpha}-a}\tau^{1-b}] ' \di \tau\\
\leq&t^{-\mu} C_{1,\alpha}^n n^{\frac{n}{\alpha}+a} \frac{1}{1-b}\left(1+\frac{1}{n-1}\right)^{\frac{n}{\alpha}+a}n^{b-1} \leq t^{-\mu} K^{n-2\delta} n^{n-\delta},
\end{aligned}
\end{equation*}
and
\begin{align*}
A_4\leq& C_{\alpha}t^{-\mu} K^{n-2\delta}n^{n-\delta} \int_{\frac{n-1}{n}}^{1} (1-\tau)^{-a}\tau^{-\frac{n}{\alpha}-b} \di \tau\\
\leq& \frac{1}{1-a}C_{\alpha}t^{-\mu} K^{n-2\delta}n^{n-\delta}\left(1+\frac{1}{n-1}\right)^{\frac{n}{\alpha}+b}\left(\frac{1}{n}\right)^{1-a}\leq t^{-\mu} K^{n-2\delta}n^{n-\delta}.
\end{align*}	
Denote 
\[
B_n(t):=\sup_{0<s\leq t}s^\mu\|\partial_x^n\rho(s)\|_{L^q},
\]
and we have the following estimate for $A_5$
\begin{equation}\label{eq:A5}
\begin{aligned}
A_5=&\int_{\frac{n-1}{n}t}^{t} \|\partial_xG_\alpha(\cdot,t-s)\ast [\partial_x^{n}\rho(s) H\rho(s)+ \rho(s)\partial_x^{n}H\rho(s)]\|_{L^q}\di s\\
\leq& C_{\alpha}\int_{\frac{n-1}{n}t}^{t} (t-s)^{-a}\|\partial_x^{n}\rho(s) H\rho(s)+ \rho(s)\partial_x^{n}H\rho(s)\|_{L^p}\di s\\
\leq &2C_{\alpha}\int_{\frac{n-1}{n}t}^{t} (t-s)^{-a}s^{-1+a}\|\partial_x^{n}\rho(s)\|_{L^q}\di s\leq 2 C_{\alpha}t^{-\mu}\int_{\frac{n-1}{n}}^{1}  (1-\tau)^{-a}\tau^{-\frac{n}{\alpha}-b} \di \tau\cdot  B_n(t)\\
\leq&2 C_{\alpha}t^{-\mu} \left(1+\frac{1}{n-1}\right)^{\frac{n}{\alpha}+b}\left(\frac{1}{n}\right)^{1-a}B_n(t)\leq \frac{1}{2}t^{-\mu}B_n(t).
\end{aligned}
\end{equation}
Combining the above estimates yields
\begin{align*}
t^\mu\|\partial_x^n\rho(t)\|_{L^q}\leq t^\mu(A_1+A_2+A_3+A_4+A_5)\leq 4K^{n-2\delta}n^{n-\delta}+\frac{1}{2}B_n(t),
\end{align*}
which gives
\[
B_n(t)\leq 8K^{n-2\delta}n^{n-\delta}\leq K^{n-\delta}n^{n-\delta}.
\]

With the results for $\frac{1}{\alpha-1}<q<\infty$, we can do the above estimates again for the cases $q=\infty$ and $q=\frac{1}{\alpha-1}$. We only need to change a little about the estimate \eqref{eq:A5} and obtain \eqref{eq:decaypower2} for the cases $q=\infty$ and $q=\frac{1}{\alpha-1}$. 

\end{proof}

\subsection{Maximum principle in $L^{p}(\mathbb{R}) ~~(p\geq 1)$ and global extension}\label{sec:global}
In this subsection, we are going to finish the proof of Theorem \ref{thm:subcritical} by extending the solutions in Theorem \ref{thm:spatialanaly} globally. We have the following maximum principle results:
\begin{lemma}\label{lmm:MaxP}
Let $\rho$ be a nonnegative strong solution to \eqref{eq:Hilbert}. For $p\geq1$, we have
\begin{align}\label{eq:maxmum}
\|\rho(t)\|_{L^p(\mathbb{R})}\leq \|\rho(s)\|_{L^p(\mathbb{R})},\quad t>s>0.
\end{align}
\end{lemma}
\begin{proof}
For $p=1$, we have 
\[
\|\rho(t)\|_{L^1}\equiv\|\rho_0\|_{L^1},\quad t>0.
\]
For $p>1$, we have
\begin{align}\label{eq:e1}
\frac{1}{p}\frac{\di}{\di t}\int_{\mathbb{R}}\rho^p(x,t)\di x=\int_{\mathbb{R}}\rho^{p-1}\partial_t\rho\di x=-\int_{\mathbb{R}}\rho^{p-1}\partial_x(\rho H\rho)\di x-\int_{\mathbb{R}}\rho^{p-1}\nu\Lambda^\alpha\rho\di x.
\end{align}
For the first term in the right hand side of \eqref{eq:e1}, we have
\begin{align*}
-\int_{\mathbb{R}}\rho^{p-1}\partial_x(\rho H\rho)\di x=&\frac{p-1}{p}\int_{\mathbb{R}}\partial_x\rho^{p}  H\rho\di x\nonumber\\
=&-\frac{p-1}{p}\int_{\mathbb{R}}\rho^{p}(x,t)\int_{\mathbb{R}}\frac{\rho(x,t)-\rho(y,t)}{|x-y|^2}\di y\di x\\
=&-\frac{p-1}{2p}\int_{\mathbb{R}}\int_{\mathbb{R}}\frac{(\rho^{p}(x,t)-\rho^{p}(y,t))(\rho(x,t)-\rho(y,t))}{|x-y|^2}\di y\di x\\
\leq&0.
\end{align*}
For the second term in the right hand side of \eqref{eq:e1}, we have
\begin{align*}
-\nu\int_{\mathbb{R}}\rho^{p-1} \Lambda^\alpha\rho\di x=&-\frac{\nu}{2}\int_{\mathbb{R}}\int_{\mathbb{R}}\frac{(\rho^{p-1}(x,t)-\rho^{p-1}(y,t))(\rho(x,t)-\rho(y,t))}{|x-y|^{1+\alpha}}\di y\di x\\
\leq&0.
\end{align*}
Combining the above two inequalities and \eqref{eq:e1}, we obtain \eqref{eq:maxmum}.
\end{proof}

\begin{theorem}\label{thm:global}
Assume $1<\alpha\leq 2$ and $0\leq \rho_0\in L^{\frac{1}{\alpha-1}}(\mathbb{R})$. Then, the local mild solution $\rho$ given by Theorem \ref{thm:contraction} can be extended globally.
\end{theorem}
\begin{proof}
Due to Theorem \ref{thm:contraction}, there exist $a>0$, $T>0$ and a unique local mild solution $\rho$ to \eqref{eq:Hilbert} in $X_T^a$ such that
\[
\sup_{0<t\leq T}t^{\frac{\alpha-1}{2\alpha}}\|G_\alpha(\cdot,t)\ast \rho_0\|_{L^{\frac{2}{\alpha-1}}(\mathbb{R}^d)}<a,\quad \sup_{0<t\leq T}t^{\frac{\alpha-1}{2\alpha}}\|\rho(t)\|_{L^{\frac{2}{\alpha-1}}}\leq 2a.
\]
Fix $0<t_0<T$, and combining Theorem \ref{thm:regularity} and Lemma \ref{lmm:MaxP} yields
\begin{align}\label{eq:claim}
\|\rho(s)\|_{L^{\frac{2}{\alpha-1}}}\leq \|\rho(t_0)\|_{L^{\frac{2}{\alpha-1}}},\quad \forall s\geq t_0.
\end{align}
From \eqref{eq:claim}  and \eqref{eq:fracdecay1} with $k=0$ and $p=q=\frac{2}{\alpha-1}$, we obtain
\[
\|G_\alpha(\cdot,t)\ast \rho(s)\|_{L^{\frac{2}{\alpha-1}}}\leq \|\rho(t_0)\|_{L^{\frac{2}{\alpha-1}}},\quad s\geq t_0.
\]
Set
\[
T_0:=\left(\frac{a}{\|\rho(t_0)\|_{L^{\frac{2}{\alpha-1}}}}\right)^{\frac{2\alpha}{\alpha-1}}.
\]
For $s\in[t_0,T]$, we have
\[
\sup_{0< t\leq T_0}t^{\frac{\alpha-1}{2\alpha}}\|G_\alpha(\cdot,t)\ast \rho(s)\|_{L^{\frac{2}{\alpha-1}} }\leq \sup_{0<t\leq T_0}t^{\frac{\alpha-1}{2\alpha}}\cdot \|\rho(t_0)\|_{L^{\frac{2}{\alpha-1}} } \leq a.
\]
Due to Theorem \ref{thm:contraction},  we can extend our solution to $s+T_0$ for any $t_0\leq s\leq T$. Moreover, this time span $T_0$ is uniform for any $s>t_0$. This proves global existence.

\end{proof}

\section{Global spatial analytic solutions  for the critical case $\alpha=1$}
In this section, we are going to prove the existence and uniqueness of spatial analytic solutions to the critical equation \eqref{eq:Hilbert} with $\alpha=1$, $\gamma>0$ for initial data $-\nu<\rho_0\in L^1(\mathbb{R})\cap H^s(\mathbb{R})$ ($s>1/2$).  When $\rho_0\geq-\mu$ for $0\leq \mu<\nu$, the  spatial analytic solutions exist at least in the time interval $(0,T)$ for $T=\frac{1}{\gamma}\ln(\frac{2\nu}{\mu}-1)$ ($T=\infty$ when $\mu=0$). When $\rho_0\geq0$, the  spatial analytic solutions exist globally and pointwisely convergent to the steady state are also obtained. 

\subsection{Well-posedness of complex Burgers equation on the upper half plane}\label{sec:analytic}
First, we derive the complex Burgers equation  from equation \eqref{eq:Hilbert} with $\alpha=1$. For $f,g\in L^{p}(\mathbb{R})$ ($p>1$), the Hilbert transform has the following properties (see e.g. \cite{Pandey}):
\[
H(Hf)=-f,\quad \partial_x(Hf)=H\partial_xf,~\textrm{ and }~H(fHg+gHf)=HfHg-fg.
\]
Applying the Hilbert transform to the  equation \eqref{eq:Hilbert} yields
\[
\partial_t(H\rho) + H\rho H\partial_x\rho - \rho \partial_x\rho-\gamma\partial_xH(\rho x) = \nu\partial_x\rho.
\]
Moreover, for $g:\mathbb{R}\to\mathbb{R}$, we have
\begin{align}\label{eq:formula}
H(xg(x))=&\frac{1}{\pi}\p\int_{\mathbb{R}}\frac{yg(y)}{x-y}\di y=\frac{1}{\pi}\p\int_{\mathbb{R}}\frac{(y-x)g(y)}{x-y}\di y+\frac{1}{\pi}\p\int_{\mathbb{R}}\frac{xg(y)}{x-y}\di y\nonumber\\
=&xHg(x)-\frac{1}{\pi}\int_{\mathbb{R}}g(x)\di x,
\end{align}
which implies
\begin{align}\label{fact2}
H(\rho x)=-\frac{\|\rho(t)\|_{L^1}}{\pi}+ux.
\end{align}
Combining the above two equations, we have
\begin{align}\label{eq:ubeta1}
\partial_tu+u\partial_xu-\rho \partial_x\rho -\gamma\partial_x(ux)= \nu\partial_x\rho.
\end{align}
Set
\[
f = u - i\rho,  \qquad u =H \rho.
\]
Hence, $f$ gives the trace of an Holomorphic function in the upper half plane. Combining \eqref{eq:Hilbert} and \eqref{eq:ubeta1} yields
\[
\partial_t f+f\partial_xf-\gamma\partial_x(fx)=i\nu\partial_x f,~~x\in\mathbb{R},~~t>0.
\]
This corresponds to the following complex equation in $\mathbb{C}_+$:
\begin{align}\label{eq:complexBurgers1}
\partial_t f+f\partial_zf-\gamma\partial_z(fz)=\partial_t f+f\partial_zf-\gamma z\partial_zf-\gamma f=i\nu\partial_z f,~~t>0.
\end{align}
By the linear transformation $g(z,t)=f(z,t)-\gamma z$, we have
\begin{multline*}
\partial_tg+g\partial_zg-\gamma^2 z=\partial_tf+(f-\gamma z)(\partial_zf-\gamma)-\gamma^2z\\
=\partial_t f+f\partial_zf-\gamma z\partial_zf-\gamma f=i\nu(\partial_z g+\gamma),
\end{multline*}
which is
\begin{align}\label{eq:complexBurgers}
\partial_tg+(g-i\nu)\partial_zg=\gamma^2 z+i\nu\gamma.
\end{align} 
Next, we derive the initial data for \eqref{eq:complexBurgers} with initial data $\rho_0$ for equation $\eqref{eq:Hilbert}$.
Let $\rho_0\in   L^1(\mathbb{R})\cap H^s(\mathbb{R})$ with $s>1/2$ be the initial data for equation \eqref{eq:Hilbert}. The initial data $\rho_0$  can be extended to a $\mathbb{C}_+$-holomorphic function by  Hilbert transform (also called Stieltjes transform, Borel transform or Markov function) for positive measures: 
\begin{align}\label{eq:initialf}
f_0(z):=\frac{1}{\pi}\int_{\mathbb{R}}\frac{\rho_0(s)}{z-s}\di s,\quad z=x+iy\in\mathbb{C}_+.
\end{align}
Direct calculation shows that
\begin{align*}
f_0(z)=\frac{1}{\pi}\int_{\mathbb{R}}\frac{\rho_0(s)}{z-s}\di s&=\frac{1}{\pi}\int_{\mathbb{R}}\frac{x-s}{y^2+(x-s)^2}\rho_0(s)\di s-i\frac{1}{\pi}\int_{\mathbb{R}}\frac{y}{y^2+(x-s)^2}\rho_0(s)\di s\\
&=:R\rho_0(x,y)-iP\rho_0(x,y),
\end{align*}
where $P\rho_0(x,y)$ and $R\rho_0(x,y)$ are given by the convolution of $\rho_0$ with the Poisson kernel and the conjugate Poisson kernel given by
\begin{align}\label{eq:PoissonKernel}
P_y(x):=\frac{1}{\pi}\frac{y}{y^2+x^2}~\textrm{ and }~R_y(x):=\frac{1}{\pi}\frac{x}{y^2+x^2}.
\end{align}
Furthermore, we have
\[
\lim_{y\to0+}[R\rho_0(x,y)-iP\rho_0(x,y)]=H\rho_0(x)-i\rho_0(x)~\textrm{for a.e.}~x\in\mathbb{R}.
\]
Hence, $f_0(x)=H\rho_0(x)-i\rho_0(x)~\textrm{for a.e.}~x\in\mathbb{R}$.
Let
\begin{align}\label{eq:initialg}
g_0(z):=f_0(z)-\gamma z,\quad z=x+iy\in \mathbb{C}_+.
\end{align}
Then, $g_0$ is a $\mathbb{C_+}$-holomorphic function.
Consider the following Cauchy problem of the Burgers type equation in $\mathbb{C}_+$:
\begin{gather}\label{eq:complexBurgers2}
\left\{
\begin{split}
&[\partial_tg+(g-i\nu)\partial_zg](z,t)=\gamma^2 z+i\nu\gamma,~~~~z=x+iy\in \mathbb{C}_+,\\
&g(z,0)=g_0(z).
\end{split}
\right.
\end{gather}
Next, we prove the existence and uniqueness of $\mathbb{C}_+$-holomorphic solutions to  \eqref{eq:complexBurgers2} by the characteristics method.
Consider the characteristics given by
\begin{equation}\label{eq:complexBurgerscharac}
\frac{\di}{\di t}Z(w,t)=g(Z(w,t),t)-i\nu,\quad Z(w,0)=w\in \mathbb{C}_+.
\end{equation}
Then,
\[
\frac{\di^2}{\di t^2 } Z(w,t)=\frac{\di}{\di t}g(Z(w,t),t)=[\partial_t g+(g-i\nu)\partial_zg](Z(w,t),t)=\gamma^2Z(w,t)+i\nu\gamma,
\]
with initial date
\[
Z(w,0)=w,\quad \frac{\di}{\di t}Z(w,t)\Big|_{t=0}=g_0(w)-i\nu,~~w\in\mathbb{C_+}.
\]
Equation \eqref{eq:complexBurgerscharac} gives the following complex trajectories:
\begin{gather}\label{eq:gloabltrajec}
Z(w,t)=\left\{
\begin{split}
&(w+i\frac{\nu}{\gamma}) \cosh \gamma t +\frac{1}{\gamma}(g_0(w)-i\nu)\sinh \gamma t-i\frac{\nu}{\gamma},\quad \gamma>0,\\ 
&(g_0(w)-i\nu)t+w=(f_0(w)-i\nu)t+w,~~\gamma=0.
\end{split}
\right.
\end{gather}
Here, we only treat the case for $\gamma>0$ and the proof of the case $\gamma=0$ is similar.  Let 
\[
Z(w,t)=Z_1(x,y,t)+iZ_2(x,y,t),~~w=x+iy\in \mathbb{C_+},
\]
and we have real part:
\begin{equation}\label{eq:realpart}
\begin{aligned}
Z_1(x,y,t)&=x \cosh  \gamma  t+ \frac{1}{\gamma} R\rho_0(x,y)\sinh \gamma t-x\sinh \gamma t\\
&=x e^{-\gamma t}+\frac{1}{\gamma} R\rho_0(x,y)\sinh \gamma t,
\end{aligned}
\end{equation}
and imaginary part:
\begin{equation}\label{eq:imaginarypart}
\begin{aligned}
Z_2(x,y,t)&=(y+\frac{\nu}{\gamma}) \cosh \gamma  t- (\frac{1}{\gamma} P\rho_0(x,y)+\frac{\nu}{\gamma})\sinh \gamma t-y\sinh \gamma t-\frac{\nu}{\gamma}\\
&= (y+\frac{\nu}{\gamma})e^{-\gamma t}-\frac{1}{\gamma} P\rho_0(x,y)\sinh \gamma t-\frac{\nu}{\gamma}.
\end{aligned}
\end{equation}
Because the initial date $g_0(w)$ in \eqref{eq:complexBurgers2} is a  $\mathbb{C}_+$-holomorphic function, $Z(w,t)$ given by \eqref{eq:gloabltrajec} is $\mathbb{C}_+$-holomorphic of $w$ for any $t\geq0$. Next, we give a lemma to show that for any fixed time $t>0$ the backward characteristics of \eqref{eq:gloabltrajec} are well defined on the set $\occ$. This result is an analogy of \cite[Lemma 2.2]{castro2008global}. We have:
\begin{lemma}\label{lmm:bijection}
Let $0\leq \mu<\nu$ and $-\mu\leq \rho_0\in  L^1(\mathbb{R})\cap H^s(\mathbb{R})$ with $s>1/2$. Denote $T=\frac{1}{\gamma}\ln(\frac{2\nu}{\mu}-1)$ ($T=\infty$ when $\mu=0$). Then
for fixed  $0<t_0<T$ and fixed $Z=Z_1+iZ_2\in \occ$, there exists a unique $w=x+iy\in \mathbb{C}_+$ such that \eqref{eq:realpart} and \eqref{eq:imaginarypart} hold. 
\end{lemma}
\begin{proof} 
Given $t_0>0$, denote
\[
a:=e^{- \gamma t_0},\quad b:=\frac{1}{\gamma}\sinh \gamma  t_0.
\]
Then \eqref{eq:realpart} and \eqref{eq:imaginarypart} become
\begin{align*}
Z_1 = ax+ b R\rho_0(x,y), \quad Z_2= ay-b P\rho_0(x,y)-(1-a)\frac{\nu}{\gamma}.
\end{align*}

\textbf{Step 1.} In this step, we prove that for any $x\in\mathbb{R}$, there exists a unique $y>0$ satisfies \eqref{eq:imaginarypart} for $Z_2\geq0$ and $0<t_0<T$.
Notice that
\begin{align}\label{eq:mu}
\frac{1-a}{b}\frac{\nu}{\gamma}=2\nu\frac{1-e^{-\gamma t_0}}{e^{\gamma t_0}-e^{-\gamma t_0}}=\frac{2\nu}{e^{\gamma t_0}+1}>\frac{2\nu}{e^{\gamma T}+1}=\mu.
\end{align}
Because $\rho_0\in L^\infty$, $P\rho_0(x,y)$ is a bounded function on $\mathbb{R}^2_+$.  By the property of Poisson kernel, we have $\lim_{y\to+\infty}P\rho_0(x,y)=0$ and hence
\begin{equation}\label{eq:proper1}
\begin{aligned}
&\lim_{y\to+\infty}Z_2(x,y,t_0)=+\infty,~~\\
&\lim_{y\to0+}Z_2(x,y,t_0)=-b \rho_0(x)+(a-1)\frac{\nu}{\gamma}=-b\left[\rho_0(x)+\frac{1-a}{b}\frac{\nu}{\gamma}\right]<0.
\end{aligned}
\end{equation}
Hence, for any fixed $Z_2\geq0$, there exists a point $y>0$ depending on $x$ such that
\[
Z_2=a(y+\frac{\nu}{\gamma})-b P\rho_0(x,y)-\frac{\nu}{\gamma}.
\]
Next, we prove the uniqueness of $y$. Suppose  there exist $y_1>y_2$ such that
\begin{align*}
Z_2=ay_1-b P\rho_0(x,y_1)+(a-1)\frac{\nu}{\gamma}=ay_2-b P\rho_0(x,y_2)+(a-1)\frac{\nu}{\gamma}.
\end{align*}
Because $P\rho_0(x,y)+\mu = P(\rho_0+\mu)(x,y)>0$, we have
\[
y_1,y_2>Z_2/a-(1-1/a)\nu/\gamma-\mu b/a,
\]
and
\[
\frac{P(\rho_0+\mu)(x,y_1)}{y_1-Z_2/a+(1-1/a)\nu/\gamma+\mu b/a}=\frac{P(\rho_0+\mu)(x,y_2)}{y_2-Z_2/a+(1-1/a)\nu/\gamma+\mu b/a}=\frac{a}{b}.
\]
Because function
\[
h(y)=\frac{y}{y-Z_2/a+(1-1/a)\nu/\gamma+\mu b/a}\cdot \frac{1}{y^2+(x-s)^2}
\]
is a decreasing function for $y>Z_2/a-(1-1/a)\nu/\gamma+\mu b/a$, we obtain a contradiction.

Now we denote by $y_{Z_2}(x)>0$ the solution of \eqref{eq:imaginarypart} with fixed $Z_2\geq0$, $t_0>0$ and $x\in\mathbb{R}$. Hence, we obtain
\begin{align}\label{eq:yx}
a(y_{Z_2}(x)+\frac{\nu}{\gamma})-\frac{\nu}{\gamma}-Z_2=bP\rho_0(x,y_{Z_2}(x)).
\end{align}

\textbf{Step 2.} In this step, we prove there exits a unique $x$ satisfies \eqref{eq:realpart} for fixed $Z_1,Z_2$ and $t_0.$
Since $\rho_0\in L^1(\mathbb{R})\cap H^s(\mathbb{R})$ ($s>1/2$), it follows that $H\rho_0\in L^\infty(\mathbb{R})$ and therefore $R\rho_0=PH\rho_0$ is a bounded function over $\mathbb{R}^2_+$. Furthermore,
\begin{align}\label{eq:proper2}
\lim_{x\to\pm\infty}[ax+bR\rho_0(x,y_{Z_2}(x))]=\pm\infty.
\end{align}
Hence, for any $Z_1\in\mathbb{R}$, we can find a $x\in\mathbb{R}$ such that
\[
Z_1=ax+bR\rho_0(x,y_{Z_2}(x)).
\]
To prove the uniqueness, we only have to prove the following function
\[
q(x)=ax+b R\rho_0(x,y_{Z_2}(x)),
\]
is an increasing function. 
Taking derivative of \eqref{eq:yx} with respect to $x$ gives
\begin{align}\label{eq:derivativeyx}
\frac{\di }{\di x}y_{Z_2}(x)= \frac{\partial_xP\rho_0(x,y_{Z_2}(x))}{a/b-\partial_yP\rho_0(x,y_{Z_2}(x))}.
\end{align}
Use \eqref{eq:derivativeyx} and the Cauchy-Riemann equations
\begin{align}\label{eq:CauchyRiemann}
\partial_x R\rho_0=-\partial_yP\rho_0,\quad \partial_x P\rho_0=\partial_yR\rho_0,
\end{align}
and taking derivative of $q(x)$ gives
\[
\frac{\di}{\di x}q(x)=\frac{b(a/b+\partial_x R\rho_0)^2+b(\partial_x P\rho_0)^2}{a/b+\partial_x R\rho_0}(x,y_{Z_2}(x)).
\]
To prove the increasing of $q(x)$, it suffices to show
\begin{align}\label{eq:increasing}
a/b +\partial_x R\rho_0(x,y)>0
\end{align}
for any $(x,y)\in\mathbb{R}^2_+$ satisfying $y>0$ and $a(y+\frac{\nu}{\gamma})-b P\rho_0(x,y)-\frac{\nu}{\gamma}\geq0$, i.e., $a(y+\frac{\nu}{\gamma})-b P(\rho_0+\mu)(x,y)-\frac{\nu}{\gamma}+b\mu\geq0$. 
We prove this by a contradiction argument. Suppose that
\[
a/b+\partial_x R\rho_0(x_0,y_0)\leq 0
\]
for some point $(x_0,y_0)\in\mathbb{R}^2_+$ with 
\begin{align}\label{eq:contradiction}
ay_0-bP(\rho_0+\mu)(x_0,y_0)\geq (1-a)\frac{\nu}{\gamma}-b\mu>0,
\end{align}
where we used \eqref{eq:mu} in the last inequality.
Due to
\[
\int_{\mathbb{R}}\frac{-y^2+s^2}{(y^2+s^2)^2}\di s=0,~~y>0,
\]
we have
\begin{align*}
-a/b&\geq \partial_x R\rho_0(x_0,y_0)=\frac{1}{\pi}\int_{\mathbb{R}}\frac{y_0^2-(x_0-s)^2}{[y_0^2+(x_0-s)^2]^2}\rho_0(s)\di s\\
&=\frac{1}{\pi}\int_{\mathbb{R}}\frac{y_0^2-(x_0-s)^2}{[y_0^2+(x_0-s)^2]^2}[\rho_0(s)+\mu]\di s>\frac{1}{\pi}\int_{\mathbb{R}}\frac{-y_0^2-(x_0-s)^2}{[y_0^2+(x_0-s)^2]^2}[\rho_0(s)+\mu]\di s\\
&=-\frac{P(\rho_0+\mu)(x_0,y_0)}{y_0},
\end{align*}
which is a contradiction with \eqref{eq:contradiction}.
\end{proof}
From the above lemma, we know that the backward characteristics are well defined on $ \occ$ in the time interval $(0,T)$. More importantly, for any $Z\in \occ$ the initial point $w$ must be interior point in $\mathbb{C}_+$.  For any $t\geq0$, we denote the backward characteristics as:
\[
Z^{-1}(\cdot,t): \occ \to \mathbb{C}_+.
\]
From the uniqueness in Lemma \ref{lmm:bijection}, $Z^{-1}(\cdot,t)$ is an $1-1$ map.

\subsection{Spatial analytic solutions to equation \eqref{eq:Hilbert} with $\alpha=1$}\label{sec:analyticity}
By Lemma \ref{lmm:bijection}, we have the following theorem which covers the results of Theorem \ref{thm:critical}:
\begin{theorem}\label{thm:analyticalpha1}
Let $0\leq \mu<\nu$ and $-\mu\leq \rho_0\in  L^1(\mathbb{R})\cap H^s(\mathbb{R})$ with $s>1/2$. Denote $T=\frac{1}{\gamma}\ln(\frac{2\nu}{\mu}-1)$. Then, we have

$\mathrm{(i)}$ The complex Burgers equation \eqref{eq:complexBurgers2} has a unique $\occ$-holomorphic solution $g(\cdot,t)$ for $t\in(0,T)$, and $\frac{\partial^k}{\partial t^k}g(\cdot,t)$ is a holomorphic function of $z$ on $\occ$ for any positive integer $k$ and $t>0$.

$\mathrm{(ii)}$ For any $t>0$, the trace of $f(z,t)=g(z,t)+\gamma z$ on the real line gives an  spatial analytic solution $\rho(x,t)\geq -\mu e^{\gamma t}$ to the equation \eqref{eq:Hilbert} with $\rho(x,0)=\rho_0(x)$ and  $\frac{\partial^k}{\partial t^k}\rho(x,t)$ is an analytic function  of $x\in\mathbb{R}$ for any positive integer $k$. Moreover, the total mass $\|\rho(t)\|_{L^1}$ is conserved:
\begin{align}\label{eq:L1conserved}
\|\rho(t)\|_{L^1}=\|\rho_0\|_{L^1}.
\end{align}

$\mathrm{(iii)}$ For $\gamma>0$ and $\mu=0$, the solution $g(z,t)$ exists globally and converges to the steady state:
\[
\lim_{t\to\infty}g(z,t)=i\nu-\sqrt{(\gamma z+i\nu)^2-2\gamma},~~\forall z\in\mathbb{C}_+,
\]
and \eqref{eq:steadypho} in Theorem \ref{thm:critical} holds.

\end{theorem}

\begin{proof}
\textbf{Step 1.} Proof of (i).
From Lemma \ref{lmm:bijection}, we have $\occ\subset\{Z(w,t):~~w\in \mathbb{C}_+\}$ and $Z^{-1}(\cdot,t)$ is well defined on $\occ$ for any fixed time $t>0$. Denote the preimage of $Z(\cdot,t)$ as:
\[
Z^{-1}(\occ,t):=\Big\{w\in \mathbb{C}_+;~~Z(w,t)\in \occ\Big\}.
\]
Denote
\[
a(t):=e^{- \gamma t},\quad b(t):=\frac{1}{\gamma}\sinh  \gamma t.
\]
For $(x,y)\in\mathbb{R}^2_+$ and $Z_2(x,y,t)\geq0$, by the Cauchy-Riemann equation \eqref{eq:CauchyRiemann},  we have
\begin{align}\label{eq:big0}
|Z_w(w,t)|=\left|\frac{\partial(Z_1,Z_2)}{\partial (x,y)}\right|(x,y)=&\left|
\begin{array}{cc}
\partial_xZ_1 & \partial_yZ_1 \nonumber\\
\partial_xZ_2 & \partial_yZ_2
\end{array}
\right|=\left|
\begin{array}{cc}
a(t)+b(t)\partial_xR\rho_0 & b(t)\partial_yR\rho_0 \nonumber\\
-b(t)\partial_xP\rho_0 & a(t)-b(t)\partial_yP\rho_0
\end{array}
\right|\\
=&\Big[a(t)+b(t)\partial_xR\rho_0\Big]^2+\Big[b(t)\partial_xP\rho_0\Big]^2\Big|_{(x,y)}>0.
\end{align}
Due to \eqref{eq:proper1} and \eqref{eq:proper2}, we obtain
\[
|Z(w,t)|\to+\infty~\textrm{ as }~|w|\to +\infty.
\]
which means $Z(\cdot,t)$ is proper \cite[Definition 6.2.2]{krantz2012implicit}.
By the Hadamard's global inverse function theorem \cite[Theorem 6.2.8]{krantz2012implicit}, there exists a inverse function $Z^{-1}(\cdot,t)$ such that
\[
Z^{-1}(\cdot,t):~\occ \to Z^{-1}(\occ,t)
\]
is a bijection. We also know $Z^{-1}$ is $\occ$-holomorphic since $Z$ is $\mathbb{C}_+$-holomorphic. Moreover, for any  $z\in \occ$, there exists $w=Z^{-1}(z,t)\in \mathbb{C}_+$. Due to $z=Z(Z^{-1}(z,t),t) \in \occ$  and $|Z_w(w,t)|\neq 0$ (by \eqref{eq:big0}), we have
\[
\partial_tZ^{-1}(z,t)=-\frac{\partial_tZ(w,t)}{\partial_wZ(w,t)},\quad w=Z^{-1}(z,t).
\]
Because of \eqref{eq:gloabltrajec}, we know $\frac{\partial^k}{\partial t^k}Z(w,t)$ is $\mathbb{C}_+$-holomorphic for any positive integer $k$.
Hence, $\frac{\partial^k}{\partial t^k}Z^{-1}(z,t)$ is $\occ$-holomorphic for any positive integer $k$.
From \eqref{eq:gloabltrajec}, we have
\begin{align}\label{eq:z}
z=(Z^{-1}(z,t)+i\frac{\nu}{\gamma}) \cosh \gamma t +\frac{1}{\gamma}(g_0(Z^{-1}(z,t))-i\nu)\sinh \gamma t-i\frac{\nu}{\gamma},\quad z\in \occ.
\end{align}
By \eqref{eq:complexBurgerscharac}, we obtain
\[
g(Z(w,t),t)=\frac{\di }{\di t}Z(w,t)+i\nu=(\gamma w+i\nu)\sinh \gamma t+(g_0(w)-i\nu)\cosh \gamma t+i\nu.
\]
Hence,
\begin{align}\label{eq:g}
g(z,t)=\gamma Z^{-1}(z,t)\sinh\gamma t+g_0(Z^{-1}(z,t))\cosh\gamma t+i\nu(1- e^{-\gamma t}),
\end{align}
which is a $\occ$-holomorphic solution to the complex Burgers equation \eqref{eq:complexBurgers2} satisfying $g(z,0)=g_0(z)$. Moreover, due to the time regularity for $Z^{-1}(z,t)$, we know that $\frac{\partial^k}{\partial t^k}g(z,t)$ is $\occ$-holomorphic for any positive integer $k$ and $t>0$.

\textbf{Step 2.} Proof of (ii).
A $\occ$-holomorphic solution to \eqref{eq:complexBurgers1}  is given by
\begin{align}\label{eq:f}
f(z,t):=g(z,t)+\gamma z,~~z\in\occ,~~t>0,
\end{align}
with initial data $f_0(z)=R\rho_0(x,y)-i P\rho_0(x,y)$, $z=x+iy\in\mathbb{C}_+$.  Combining \eqref{eq:z} and \eqref{eq:g}, we obtain for $z\in\occ$:
\begin{equation}\label{eq:fz}
\begin{aligned}
z&= e^{-\gamma t}Z^{-1}(z,t)+\frac{1}{\gamma}f_0(Z^{-1}(z,t))\sinh\gamma t+i\frac{\nu}{\gamma}(e^{-\gamma t}-1),\\
f(z,t)&= f_0(Z^{-1}(z,t))e^{\gamma t}.
\end{aligned}
\end{equation}
Consider the trace of $f(z,t)$ on the real line and define:
\[
f(x,t)=:u(x,t)-i \rho(x,t).
\]
Due to Lemma \ref{lmm:bijection}, for any $x\in\mathbb{R}$, we have $Z^{-1}(x,t)=:a_x+ib_x\in\mathbb{C_+}$ with some positive real number $b_x>0$. From \eqref{eq:fz}, we have
\[
f(x,t)= f_0(a_x+ib_x)e^{\gamma t}=R\rho_0(a_x,b_x)e^{\gamma t}-i P\rho_0(a_x,b_x)e^{\gamma t}
\]
Therefore,
\begin{align}
\rho(x,t)=P\rho_0(a_x,b_x)e^{\gamma t}=P(\rho_0+\mu)(a_x,b_x)e^{\gamma t}-\mu e^{\gamma t}\geq -\mu e^{\gamma t},~~x\in\mathbb{R}.
\end{align}
Hence, $\rho(x,t)$ is an  spatial analytic solution of \eqref{eq:Hilbert}. Moreover, by the uniqueness of solutions to the characteristics equation \eqref{eq:complexBurgers2} we know  solutions to equation \eqref{eq:Hilbert} is unique.

\textbf{Step 3.} The proof of (iii) follows from the method in  \cite{rogers1993interacting} and we put it into appendix \ref{apd:convergence}. 
\end{proof}

\begin{remark}
\begin{enumerate}
\item When $\nu=0$, $\rho_\infty$ given by \eqref{eq:steadypho} reduces to \cite[Eq. (2.15)]{gao2020large}. For $\gamma=0$, we have $\rho_\infty=0.$
\item Comparing with Theorem \ref{thm:analyticalpha1}, \cite[Theorem 4.1]{castro2008global} and \cite[Theorem 4.8]{castro2008global}, a nature conjecture is that the $\|\partial_xH\rho\|_{L^\infty}$ blows up in finite time when initial data satisfies $\rho_0(x_0)<0$ for some $x_0\in\mathbb{R}$. According to \cite[Remark 2.1]{gao2020large}, the blow-up behavior  is much more complicated for $\gamma>0$ and $\nu=0$ (blow-up along a curve),  while the the blow-up behavior for $\gamma=\nu=0$ is simpler (blow up along a straight line).

\end{enumerate}
\end{remark}

\noindent\textbf{Acknowledgements} 
X. Xue was supported by the Natural Science Foundation of China (grants 11731010 and 11671109).

\appendix

\section{Proof of time continuity of $S\rho(t)$ in Theorem \ref{thm:contraction}}\label{apd:continuity}
\begin{proof}
Since the first part $G_\alpha(\cdot,t)\ast \rho_0$ corresponds to the solution of fractional heat equation, it is continuous with respect to $t$ in space $L^{\frac{1}{\alpha-1}}(\mathbb{R})$. Hence, we only need to show continuity of the second term 
\[
\rho_2(x,t):=\int_0^t\partial_xG_\alpha(\cdot,t-s)\ast(\rho(s)H\rho(s))\di s.
\]
Let $t>\tau>0$ and we have
\begin{multline*}
\|\rho_2(t)-\rho_2(\tau)\|_{L^{\frac{1}{\alpha-1}}}\leq \left\|\int_\tau^t\partial_xG_\alpha(\cdot,t-s)\ast(\rho(s)H\rho(s))\di s\right\|_{L^{\frac{1}{\alpha-1}}}\\
+\left\|\int_0^\tau\partial_xG_\alpha(\cdot,t-s)\ast(\rho(s)H\rho(s))-\partial_xG_\alpha(\cdot,\tau-s)\ast(\rho(s)H\rho(s))\di s\right\|_{L^{\frac{1}{\alpha-1}}}=:I_1+I_2.
\end{multline*}
For $I_1$, we have
\begin{align*}
I_1=\left\|\int_\tau^t\partial_xG_\alpha(\cdot,t-s)\ast(\rho(s)H\rho(s))\di s\right\|_{L^{\frac{1}{\alpha-1}}}\leq Ca^2\int_{\tau/t}^1(1-s)^{-\frac{1}{\alpha}}s^{-\frac{\alpha-1}{\alpha}}\di s\to 0~\textrm{ as }~t\to\tau.
\end{align*}
For $I_2$, set $g(x,s):=\partial_xG_\alpha(\cdot,\tau-s)\ast(\rho(s)H\rho(s))$ for $0<s<\tau$, and then
\[
\partial_xG_\alpha(\cdot,t-s)\ast(\rho(s)H\rho(s))=G_\alpha(\cdot,t-\tau)\ast g(s).
\]
We have
\begin{equation}\label{eq:I2}
\begin{aligned}
I_2=\left\|\int_0^\tau G_\alpha(\cdot,t-\tau)\ast g(s)-g(s)\di s\right\|_{L^{\frac{1}{\alpha-1}}}\leq\int_0^\tau\left\| G_\alpha(\cdot,t-\tau)\ast g(s)-g(s)\right\|_{L^{\frac{1}{\alpha-1}}}\di s.
\end{aligned}
\end{equation}
Next, we estimate the integrand $\|G_\alpha(\cdot,t-\tau)\ast g(s)-g(s)\|_{L^{\frac{1}{\alpha-1}}}$. For arbitrary $r>0$, by Jensen's inequality we have 
\begin{equation}\label{eq:sindpendent}
\begin{aligned}
\|G_\alpha(\cdot,t-\tau)\ast &g(s)-g(s)\|_{L^{\frac{1}{\alpha-1}}}^{\frac{1}{\alpha-1}}=\int_{\mathbb{R}}\left|\int_{\mathbb{R}}G_\alpha(x-y,t-\tau)[g(y,s)-g(x,s)]\di y\right|^{\frac{1}{\alpha-1}}\di x\\
\leq&\int_{\mathbb{R}}\int_{\mathbb{R}}G_\alpha(x-y,t-\tau)|g(y,s)-g(x,s)|^{\frac{1}{\alpha-1}}\di y\di x\\
\leq &\int_{\mathbb{R}}\int_{B(x,r)}G_\alpha(x-y,t-\tau)|g(y,s)-g(x,s)|^{\frac{1}{\alpha-1}}\di y\di x\\
&+\int_{\mathbb{R}}\int_{\mathbb{R}\setminus B(x,r)}G_\alpha(x-y,t-\tau)|g(y,s)-g(x,s)|^{\frac{1}{\alpha-1}}\di y\di x=:I_{21}+I_{22}.
\end{aligned}
\end{equation}
For $I_{21}$, we have
\begin{equation}\label{eq:firstterm}
\begin{aligned}
I_{21}=&\int_{\mathbb{R}}\int_{B(x,r)}G_\alpha(x-y,t-\tau)|g(y,s)-g(x,s)|^{\frac{1}{\alpha-1}}\di y\di x\\
=&\int_{\mathbb{R}}\int_{B(0,r)}G_\alpha(z,t-\tau)|g(x+z,s)-g(x,s)|^{\frac{1}{\alpha-1}}\di z\di x\\
=&\int_{B(0,r)}G_\alpha(z,t-\tau)\int_{\mathbb{R}}|g(x+z,s)-g(x,s)|^{\frac{1}{\alpha-1}}\di x\di z\\
\leq&\int_{B(0,r)}G_\alpha(z,t-\tau)\sup_{|h|\leq r}\int_{\mathbb{R}}|g(x+h,s)-g(x,s)|^{\frac{1}{\alpha-1}}\di x\di z\\
\leq &\sup_{|h|\leq r}\int_{\mathbb{R}}|g(x+h,s)-g(x,s)|^{\frac{1}{\alpha-1}}\di x.
\end{aligned}
\end{equation}
Notice that $\partial_xG_\alpha(x,t)=t^{-\frac{2}{\alpha}}\partial_xG_\alpha(t^{-\frac{1}{\alpha}}x,1)$. Denote
\[
f(x)=(\tau-s)^{-\frac{2}{\alpha}}|\partial_xG_\alpha\big((\tau-s)^{-\frac{1}{\alpha}}(x+h),1\big)-\partial_xG_\alpha\big((\tau-s)^{-\frac{1}{\alpha}}x,1\big)|.
\]
By the definition of $g$, we have
\begin{equation}\label{eq:sindependet1}
\begin{aligned}
I_{21}^{\alpha-1}\leq& \sup_{|h|\leq r}\|f\ast (\rho(s)H\rho(s))\|_{L^{\frac{1}{\alpha-1}}}\leq \sup_{|h|\leq r}\|f\|_{L^1}\|\rho(s)H\rho(s)\|_{L^{\frac{1}{\alpha-1}}}\\
\leq&Ca^2(\tau-s)^{-\frac{1}{\alpha}}s^{-\frac{\alpha-1}{\alpha}}\sup_{|h|\leq r}\left\| \partial_xG_\alpha(\cdot+h,1)-\partial_xG_\alpha(\cdot,1)\right\|_{L^1}.
\end{aligned}
\end{equation}
Due to $G_\alpha(x,t)=t^{-\frac{1}{\alpha}}G(t^{-\frac{1}{\alpha}}x,1)$, we obtain
\begin{equation}\label{eq:secondterm}
\begin{aligned}
I_{22}=&\int_{\mathbb{R}}\int_{\mathbb{R}\setminus B(x,r)}G_\alpha(x-y,t-\tau)|g(y,s)-g(x,s)|^{\frac{1}{\alpha-1}}\di y\di x\\
\leq& \int_{\mathbb{R}}\int_{\mathbb{R}\setminus B(0,r/(t-\tau)^{\frac{1}{\alpha}})}G_\alpha(z,1)|g(x+(t-\tau)^{\frac{1}{\alpha}}z,s)-g(x,s)|^{\frac{1}{\alpha-1}}\di z\di x\\
\leq &2\|g(s)\|_{L^{\frac{1}{\alpha-1}}}^{\frac{1}{\alpha-1}}\int_{\mathbb{R}\setminus B(0,r/(t-\tau)^{\frac{1}{\alpha}})}G_\alpha(z,1)\di z.
\end{aligned}
\end{equation}
From \eqref{eq:kell2}, we know $\|\partial_x G_\alpha(\cdot,t)\|_{L^1}=t^{-\frac{1}{\alpha}}\|\partial_xG_\alpha(\cdot,1)\|_{L^1}.$
By Young's convolution inequality, we obtain
\begin{equation}\label{eq:sindependet0}
\begin{aligned}
\|g(s)\|_{L^{\frac{1}{\alpha-1}}}\leq &\|\partial_xG_\alpha(\cdot,\tau-s)\|_{L^1}\|\rho(s)H\rho (s)\|_{L^{\frac{1}{\alpha-1}}}\\
\leq& Ca^2(\tau-s)^{-\frac{1}{\alpha}}s^{-\frac{\alpha-1}{\alpha}}.
\end{aligned}
\end{equation}
Combining \eqref{eq:I2}-\eqref{eq:sindependet0}, we obtain
\begin{equation}
\begin{aligned}
I_2\leq &\int_0^\tau\|G_\alpha(\cdot,t-\tau)\ast g(s)-g(s)\|_{L^{\frac{1}{\alpha-1}}}\di s\leq \int_0^\tau(I_{21}+I_{22})^{\alpha-1}\di s\\
\leq &Ca^2\int_0^\tau (\tau-s)^{-\frac{1}{\alpha}}s^{-\frac{\alpha-1}{\alpha}}\di s\sup_{|h|\leq r}\|\partial_xG_\alpha(\cdot+h,1)-\partial_xG_\alpha(\cdot,1)\|_{L^1}\\
&\qquad\qquad+Ca^2\int_0^\tau (\tau-s)^{-\frac{1}{\alpha}}s^{-\frac{\alpha-1}{\alpha}}  \di s \left( \int_{\mathbb{R}\setminus B(0,r/(t-\tau)^{\frac{1}{\alpha}})}G_\alpha(z,1)\di z\right)^{\alpha-1}.
\end{aligned}
\end{equation}
By \cite[Lemma 4.3]{brezis2010functional}, letting $t\to\tau$ first and then $r\to0$, we have $I_2\to0$.
\end{proof}

\section{Proof of \eqref{eq:claim1}}\label{app:proofclaim1}
\begin{proof}[Proof of \eqref{eq:claim1}]
We prove \eqref{eq:claim1} for $n$ big enough. Notice that
\[
\mu=\frac{n}{\alpha}+a+b-1.
\]
We have
\begin{align*}
&\sum_{m=2}^{n-1} n^{\mu} m^{m-\delta}  \left(\frac{1}{m-1}\right)^{\frac{m}{\alpha}+b}=\sum_{m=2}^{n-1} m^{m-\delta}  \left(\frac{1}{m-1}\right)^{\frac{m}{\alpha}+b}n^{\frac{n}{\alpha}+a+b-1}\\
=&\sum_{m=2}^{n-1} \left(\frac{m}{n}\right)^{m}\left(\frac{n}{m-1}\right)^{\frac{m}{\alpha}}\frac{1}{m^\delta (m-1)^b}n^{\frac{n}{\alpha}+a+b-1+m-\frac{m}{\alpha}}.
\end{align*}
Because there exists some constant $M$ independent of $n$ such that
\[
\left(\frac{m}{n}\right)^{m}\left(\frac{n}{m-1}\right)^{\frac{m}{\alpha}}\leq \left(\frac{m}{n}\right)^{\frac{m}{\alpha}}\left(\frac{n}{m-1}\right)^{\frac{m}{\alpha}}\leq \left(1+\frac{1}{m-1}\right)^{\frac{m}{\alpha}}\leq M,
\]
we have
\begin{align*}
&\sum_{m=2}^{n-1} n^{\mu} m^{m-\delta}  \left(\frac{1}{m-1}\right)^{\frac{m}{\alpha}+b}\leq M\sum_{m=2}^{n-1} \frac{1}{m^\delta (m-1)^b}n^{\frac{n}{\alpha}+a+b-1+m-\frac{m}{\alpha}}\\
=& Mn^{n-\delta}\sum_{m=2}^{n-1} \frac{1}{m^\delta (m-1)^b}n^{-(n-m)(1-\frac{1}{\alpha})+a+b-1+\delta}\\
=& Mn^{n-\delta}\left[\sum_{2\leq m\leq \frac{n}{2}} \frac{n^{a+b-1+\delta}}{m^\delta (m-1)^b}n^{-(n-m)(1-\frac{1}{\alpha})}+\sum_{\frac{n}{2}<m\leq n-1} \frac{n^{a+b-1+\delta}}{m^\delta (m-1)^b}n^{-(n-m)(1-\frac{1}{\alpha})}\right]\\
=:&Mn^{n-\delta}(I_1+I_2).
\end{align*}
To prove \eqref{eq:claim1}, it suffices to prove $I_1\leq \frac{1}{2M}$ and $I_2\leq \frac{1}{2M}$ for $n$ big enough. For simplicity, we assume $n$ be an even number. For $I_1$, we have
\begin{align*}
I_1=\sum_{2\leq m\leq \frac{n}{2}} \frac{n^{a+b-1+\delta}}{m^\delta (m-1)^b}n^{-(n-m)(1-\frac{1}{\alpha})}\leq n^{a+b-1+\delta}\sum_{2\leq m\leq \frac{n}{2}} n^{-(n-m)(1-\frac{1}{\alpha})}\leq \frac{2n^{a+b-1+\delta}}{n^{\frac{n}{2}(1-\frac{1}{\alpha})}}\leq \frac{1}{2M}.
\end{align*}
For $I_2$, we have
\begin{align*}
I_2=&\sum_{\frac{n}{2}<m\leq n-1} \frac{n^{a+b-1+\delta}}{m^\delta (m-1)^b}n^{-(n-m)(1-\frac{1}{\alpha})}\leq \sum_{\frac{n}{2}<m\leq n-1} \frac{2^{\delta+b} n^{a+b-1+\delta}}{n^\delta (n-2)^b}n^{-(n-m)(1-\frac{1}{\alpha})}\\
\leq& \frac{2^{\delta+b} n^{a+b-1}}{(n-2)^b} \sum_{\frac{n}{2}<m\leq n-1}n^{-(n-m)(1-\frac{1}{\alpha})}.
\end{align*}
For $n$ big enough, we have
\[
\sum_{\frac{n}{2}<m\leq n-1}n^{-(n-m)(1-\frac{1}{\alpha})}\leq \frac{2}{n^{1-\frac{1}{\alpha}}},
\]
and hence 
\[
I_2\leq 2^{\delta+b+1}\left(\frac{n}{n-2}\right)^b\frac{ 1}{n^{1-a}n^{1-\frac{1}{\alpha}}} \leq \frac{1}{2M}.
\]

\end{proof}

\section{Proof of  Theorem \ref{thm:analyticalpha1} (\lowercase{$\mathrm{iii}$})}\label{apd:convergence}
\begin{proof}[Proof of Theorem \ref{thm:analyticalpha1} (iii)]
to prove  the convergence result (iii). Recall formula \eqref{eq:fz}.
For fixed $z\in\occ$, denote 
\[
e^{-\gamma t}Z^{-1}(z,t)=:z_r(t)+iz_i(t).
\]
Next, we prove that $z_r(t)+iz_i(t)$ converges to a point $w=z_r^*+iz_i^*\in \mathbb{C}_+$ as $t\to\infty$. To this end, we first prove
$|z_r(t)|$ and $z_i(t)$ are all bounded from above and below uniformly in time $t$.

Because
\[
f_0(Z^{-1}(z,t))=R\rho_0(e^{\gamma t}z_r(t),e^{\gamma t}z_i(t))-i P\rho_0(e^{\gamma t}z_r(t),e^{\gamma t}z_i(t)),
\]
by \eqref{eq:fz}, we have
\begin{multline}\label{eq:z1}
z=z_r(t)+R\rho_0(e^{\gamma t}z_r(t),e^{\gamma t}z_i(t))\frac{\sinh\gamma t}{\gamma}\\
+i\left[z_i(t)-P\rho_0(e^{\gamma t}z_r(t),e^{\gamma t}z_i(t))\frac{\sinh\gamma t}{\gamma}+\frac{\nu}{\gamma}(e^{-\gamma t}-1)\right].
\end{multline}
Due to $-P\rho_0(e^{\gamma t}z_r(t),e^{\gamma t}z_i(t))\frac{\sinh\gamma t}{\gamma}+\frac{\nu}{\gamma}(e^{-\gamma t}-1)\leq 0$, we have
\[
z_i(t)\geq\Im(z)>0.
\]
Moreover, we have
\begin{align*}
\Im(z)&=z_i(t)-P\rho_0(e^{\gamma t}z_r(t),e^{\gamma t}z_i(t))\frac{\sinh\gamma t}{\gamma}+\frac{\nu}{\gamma}(e^{-\gamma t}-1)\\
&=z_i(t)-\int_{\mathbb{R}}\frac{e^{\gamma t}z_i(t)}{e^{2\gamma t}z_i^2(t)+(e^{\gamma t}z_r(t)-s)^2}\rho_0(s)\di s\frac{\sinh\gamma t}{\gamma}+\frac{\nu}{\gamma}(e^{-\gamma t}-1)\\
&\geq z_i(t)-\frac{1}{\gamma}\int_{\mathbb{R}}\frac{e^{2\gamma t}z_i(t)}{2e^{2\gamma t}z_i^2(t)+2(e^{\gamma t}z_r(t)-s)^2}\rho_0(s)\di s+\frac{\nu}{\gamma}(e^{-\gamma t}-1)\\
&\geq z_i(t)-\frac{1}{2\gamma z_i(t)}+\frac{\nu}{\gamma}(e^{-\gamma t}-1),
\end{align*}
which implies
\[
z_i(t)\leq \Im(z)+1/\sqrt{2\gamma}-\frac{\nu}{\gamma}(e^{-\gamma t}-1).
\]
Hence, $z_i(t)$ is bounded as
\[
0<\Im(z)\leq z_i(t)\leq \Im(z)+1/\sqrt{2\gamma}-\frac{\nu}{\gamma}(e^{-\gamma t}-1).
\]
Next, we prove
\[
\sup_{t\geq0}|z_r(t)|<+\infty.
\]
We prove this by a contradiction argument. If there exists $t_n\to\infty$ such that $z_r(t_n)\to\infty$, then by the dominated convergence theorem we have
\[
R\rho_0(e^{\gamma t_n}z_r(t_n),e^{\gamma t_n}z_i(t_n))\frac{\sinh \gamma t_n}{\gamma}=\frac{1}{\pi}\int_{\mathbb{R}}\frac{(e^{\gamma t_n}z_r(t_n)-s)\rho_0(s)}{e^{2\gamma t_n}z_i^2(t_n)+(e^{\gamma t_n}z_r(t_n)-s)^2}\di x\frac{\sinh \gamma t_n}{\gamma}\to0,~~n\to\infty.
\]
By \eqref{eq:z1}, we obtain a contradiction that
\[
\Re(z)=z_r(t_n)+R\rho_0(e^{\gamma t_n}z_r(t_n),e^{\gamma t_n}z_i(t_n))\frac{\sinh \gamma t_n}{\gamma}\to\infty.
\]

Since  $|z_r(t)|$ and $z_i(t)$ are bounded, there exist $t_n\to\infty$  and two constant $z_r^*$, $z_i^*>0$ such that
\[
z_r(t_n)\to z_r^*,~~z_i(t_n)\to z_i^*,~~n\to\infty.
\]
For any $s\in\mathbb{R}$, we have
\[
\frac{e^{\gamma t_n}z_r(t_n)-s}{e^{2\gamma t_n}z_i^2(t_n)+(e^{\gamma t_n}z_r(t_n)-s)^2}\sinh\gamma t_n\to\frac{z_r^*}{2(z_i^*)^2+2(z_r^*)^2},~~n\to\infty.
\]
Then, by the dominated convergence theorem we have
\begin{align*}
&\lim_{n\to\infty}R\rho_0(e^{\gamma t_n}z_r(t_n),e^{\gamma t_n}z_i(t_n))\frac{\sinh\gamma t_n}{\gamma}\\
=&\frac{1}{\gamma\pi}\lim_{n\to\infty}\int_{\mathbb{R}}\frac{e^{\gamma t_n}z_r(t_n)-s}{e^{2\gamma t_n}z_i^2(t_n)+(e^{\gamma t_n}z_r(t_n)-s)^2}\rho_0(s)\di s\sinh\gamma t_n\\
=&\frac{1}{2\gamma\pi}\frac{z_r^*}{(z_i^*)^2+(z_r^*)^2}.
\end{align*}
Similarly, we have
\begin{align*}
\lim_{n\to\infty}P\rho_0(e^{\gamma t_n}z_r(t_n),e^{\gamma t_n}z_i(t_n))\frac{\sinh \gamma t_n}{\gamma}=\frac{1}{2\gamma\pi}\frac{z_i^*}{(z_i^*)^2+(z_r^*)^2}.
\end{align*}
Let $w:=z_r^*+iz_i^*$. Then, let $t=t_n$ in \eqref{eq:z1} and sending $n\to\infty$ gives
\[
z=w+\frac{1}{2\gamma\pi}\frac{z_r^*-iz_i^*}{(z_i^*)^2+(z_r^*)^2}-i\frac{\nu}{\gamma}=w+\frac{1}{2\gamma\pi w}-i\frac{\nu}{\gamma}.
\]
We obtain a unique solution in $\mathbb{C_+}$ (with positive imaginary part):
\[
w=\frac{1}{\gamma\pi z+i\nu\pi-\sqrt{(\gamma\pi z+i\nu\pi)^2-2\gamma\pi}}.
\]
Hence, we have
\[
e^{-\gamma t}Z^{-1}(z,t)=z_r(t)+iz_i(t)\to w=\frac{1}{\gamma\pi z+i\nu\pi-\sqrt{(\gamma\pi z+i\nu\pi)^2-2\gamma\pi}},~~t\to\infty.
\]
By \eqref{eq:fz} and using the dominated convergence theorem again, we have
\begin{align*}
f(z,t)=&f_0(Z^{-1}(z,t))e^{\gamma t}\\
=&\int_{\mathbb{R}}\frac{e^{2\gamma t}z_r(t)-s}{e^{2\gamma t}z_i^2(t)+[e^{\gamma t}z_r(t)-s]^2}\rho_0(s)\di s-i\int_{\mathbb{R}}\frac{e^{2\gamma t}z_i(t)}{e^{2\gamma t}z_i^2(t)+[e^{\gamma t}z_r(t)-s]^2}\rho_0(s)\di s\\
&\to \frac{z_r^*-iz_i^*}{(z_i^*)^2+(z_r^*)^2}=\frac{1}{w}=\gamma\pi z+i\nu\pi-\sqrt{(\gamma\pi z+i\nu\pi)^2-2\gamma\pi},\quad t\to\infty.
\end{align*}
Let $z=x+iy$, $y\geq0$ and the imaginary part is given by
\begin{multline}
\Im(\frac{1}{w})=\pi\nu\\
-\frac{\sqrt{\sqrt{[\pi^2\gamma^2x^2-\pi^2(\gamma y+\nu)^2-2\pi\gamma]^2+4\pi^4\gamma^2x^2(\gamma y+\nu)^2}-[\pi^2\gamma^2x^2-\pi^2(\gamma y+\nu)^2-2\pi\gamma]}}{\sqrt{2}}<0.
\end{multline}
Consider the trace on the real line $x\in\mathbb{R}$ and $y=0$, and we obtain
\begin{align*}
\rho(x,t)=-\Im(f(x,t))\to\rho_\infty(x),\quad t\to \infty.
\end{align*}
For $\nu=0$, it is the same as in \cite[Eq. (2.15)]{gao2020large}. For $\gamma=0$, we have $\rho_\infty=0.$
\end{proof}

\bibliographystyle{plain}
\bibliography{bibofHilbert}

\end{document}